\documentclass[12pt,a5paper]{article}

\usepackage{euler,amsmath,amsfonts,defns}

\usepackage{ajr2}
\usepackage[bookmarks=false]{hyperref}
\usepackage{color,graphicx}

\usepackage[3D]{movie15}
\allowdisplaybreaks

\newcommand{\icc}{\textsc{icc}s}
\newcommand{\Res}{\operatorname{Res}}
\newcommand{\KS}{Ku\-ra\-mo\-to--Siva\-shin\-sky}

\newcommand{\bbE}{\mathbb E}
\newcommand{\uij}{U_{i,j}}
\newcommand{\dotuij}{\dot U_{i,j}}

\newcommand{\cL}{\mathcal L}
\newcommand{\shift}{\text{\small$\mathcal E$}}
\newcommand{\glpde}{Ginzburg--Landau \pde~\eqref{E_gl2d}}
\newcommand{\bdelta}{\text{\boldmath$\delta$}}

\newcommand{\pat}{\hat} 

\title{Accurate macroscale modelling of spatial dynamics in multiple dimensions}

\author{
A.~J. Roberts\thanks{Corresponding author: School of Mathematical Sciences, University of Adelaide, South Australia~5005, Australia. 
\protect\url{mailto:anthony.roberts@adelaide.edu.au}}
\and
Tony MacKenzie\thanks{Department of Mathematics and Computing, University of Southern Queensland, Toowoomba, Queensland~4352, Australia.} 
\and 
J.~E. Bunder\thanks{School of Mathematical Sciences, University of Adelaide, South Australia~5005, Australia. 
\protect\url{mailto:judith.bunder@adelaide.edu.au}}
}

\begin{document}

\maketitle

\begin{abstract}
Developments in dynamical systems theory provides new support for the macroscale modelling of \pde{}s and other microscale systems such as Lattice Boltzmann, Monte Carlo or Molecular Dynamics simulators.  By systematically resolving subgrid microscale dynamics the dynamical systems approach constructs accurate closures of macroscale discretisations of the microscale system.  Here we specifically explore reaction-diffusion problems in two spatial dimensions as a prototype of generic systems in multiple dimensions.  Our approach unifies into one the modelling of systems by a type of finite elements, and the `equation free' macroscale modelling of microscale simulators efficiently executing only on small patches of the spatial domain. Centre manifold theory ensures that a closed model exist on the macroscale grid, is emergent, and is systematically approximated.  Dividing space either into overlapping finite elements or into spatially separated small patches, the specially crafted inter-element\slash patch coupling also ensures that the constructed discretisations are consistent with the microscale system\slash\pde\ to as high an order as desired.    Computer algebra handles the considerable algebraic details as seen in the specific application to the Ginzburg--Landau \pde.  However, higher order models in multiple dimensions require a mixed numerical and algebraic approach that is also developed.  The modelling here may be straightforwardly adapted to a wide class of reaction-diffusion \pde{}s and lattice equations in multiple space dimensions.  When applied to patches of microscopic simulations our coupling conditions promise efficient macroscale simulation.
\end{abstract}

 

\tableofcontents

\section{Introduction}

Computational simulation is a key enabling technology in engineering, science and other quantitative fields \cite[e.g.]{Dolbow04}.
Coherent spatio-temporal dynamics is the preeminent example of complex system behaviour as it emerges from the interactions of many similar components at each locale in space~\cite[e.g.]{Greenside96, Louzoun01, Hyman86a, Greenside84}.
We must simulate such systems on the scale of interest and operation.
But systems that depend on physical processes over multiple scales pose notorious difficulties.
These multiscale difficulties are major obstacles to progress in fields as diverse as environmental and geosciences, climate, materials, combustion, high energy density physics, fusion, bioscience, chemistry, power grids and information networks~\cite{Dolbow04}.
Following the `equation free' approach of Kevrekidis, Samaey and colleagues~\cite[e.g.]{Kevrekidis09a}, we here address the extraction, using dynamical systems theory, of computationally efficient macroscale models from given microscopic models, whether \pde\ or lattice dynamics or other microscale simulators. 

The ultimate aim of this article is to provide theoretical support for and to further develop Kevrekidis' et al.~\cite{Kevrekidis03b, Kevrekidis09a} `equation free' approach to multiscale modelling.
Given a numerical simulator for physical components at much smaller scales than the scale of primary interest, the `equation free' aim is to bridge space and time scales to simulations resolving the macroscale of primary interest.
Here we bridge space scales by generalising  to multiple dimensions (specifically~2D) both the methodology and supporting theory for the `equation free', gap-tooth method for microsimulators that was initiated by Gear, Li \& Kevrekidis~\cite{Gear03} and Samaey, Kevrekidis \& Roose~\cite{Samaey03a, Samaey03b}.
Figures~\ref{fig:Matlab/early5-4-8-12-10D/0}--\ref{fig:Matlab/early5-4-8-12-10D/1} show snapshots of an example simulation of the gap-tooth method in~2D:  microscale simulators executing on a fine grid \emph{only within} the $64$~patches (of an $8\times 8$ macroscale grid) are coupled across empty space to efficiently simulate the dynamics over large spatial scales; the computational cost here is one sixth that of a microscale simulation over the whole domain, but much better gains may be obtained (but do not provide suitable graphics).
Crucially, although our analysis considers systems in principle representable in the class of the reaction-diffusion \pde~\eqref{Erde}, in application the gap-tooth method does not require knowledge of the specific \pde: the gap-tooth method provides \emph{on the fly closure}.  Such closure constitutes critical components of, for example, mathematical homogenization~\cite[e.g.]{Samaey03b, Gustafsson03, Balakotaiah03}, renormalization group techniques~\cite[e.g.]{Ei99, Mudavanhu03, Chorin05}, and multiscale finite elements~\cite[e.g.]{Hou97, Chen02}.  But by avoiding the need to algebraically find the closure the gap-tooth scheme has the potential to efficiently simulate many systems whose macroscale dynamics are otherwise unaccessible.

\newcommand{\patchfigure}[4]{
\begin{figure}
\centering
\includegraphics{#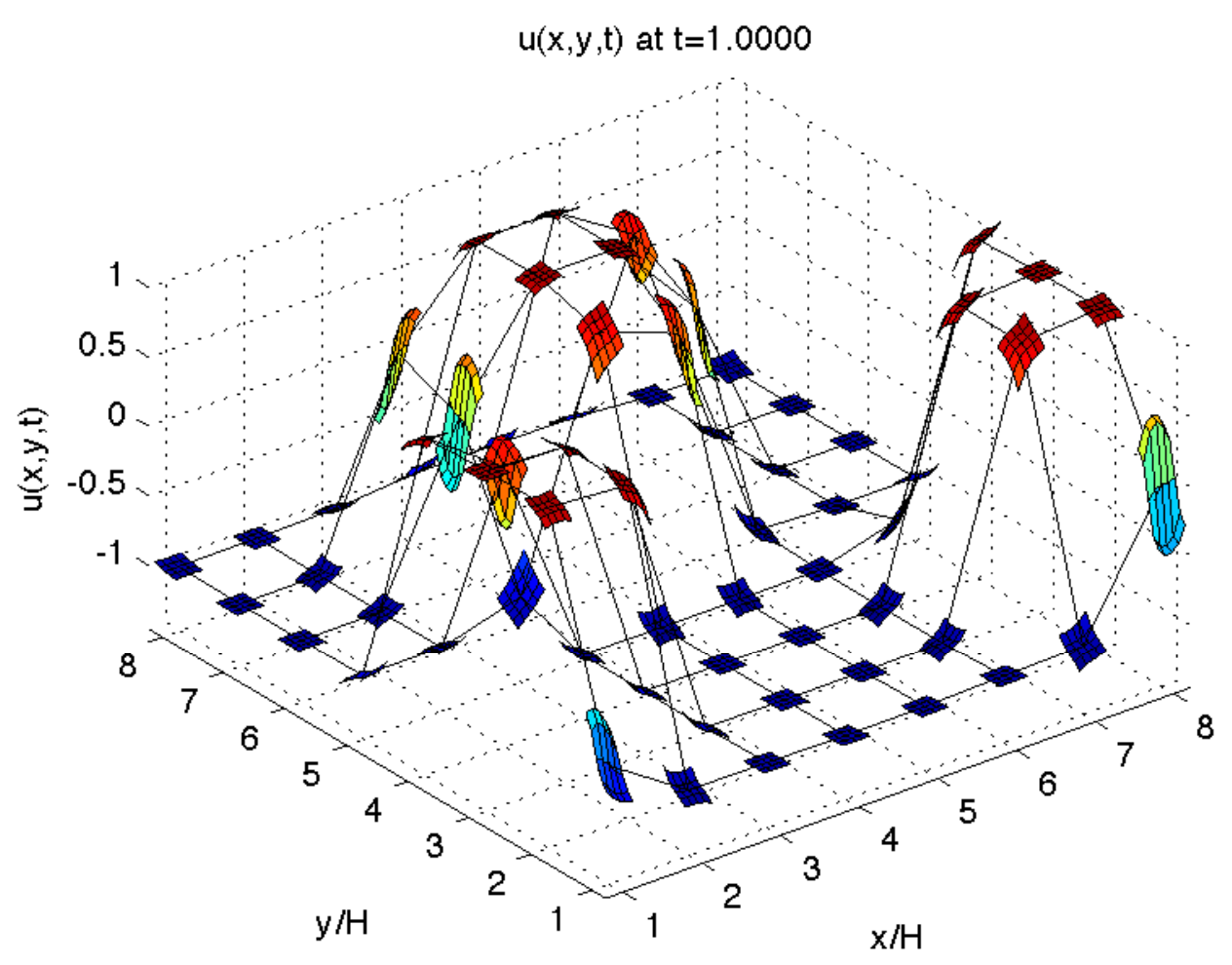}\\
\fbox{\includemovie[poster,toolbar,label=MatlabArxiv/early5-4-#2-12-10D/#1.u3d,text=(view u3d in Adobe Reader), 3Drender=SolidWireframe,
3Daac=7, 3Droll=0, 3Dc2c=-1 -1.4 0.7, 3Droo=#200, 3Dcoo=0 0 0, 3Dlights=CAD]
{.5\linewidth}{.5\linewidth}{#1.u3d}}
\parbox[b]{0.45\linewidth}{
\caption{perspective view of $#2\times#2$ patches of a microsimulator on a macroscale grid of spacing $H=9.6$ at nondimensional time $t=#3$\,.  Each patch is a microscale discretisation of the \glpde{} with nonlinearity $\alpha=3$ on a $5\times5$ fine grid: the microscale simulator executes on only~$16$\% of space.  #4}
\label{fig:Matlab/early5-4-#2-12-10D/#1}}
\end{figure}
}

\patchfigure{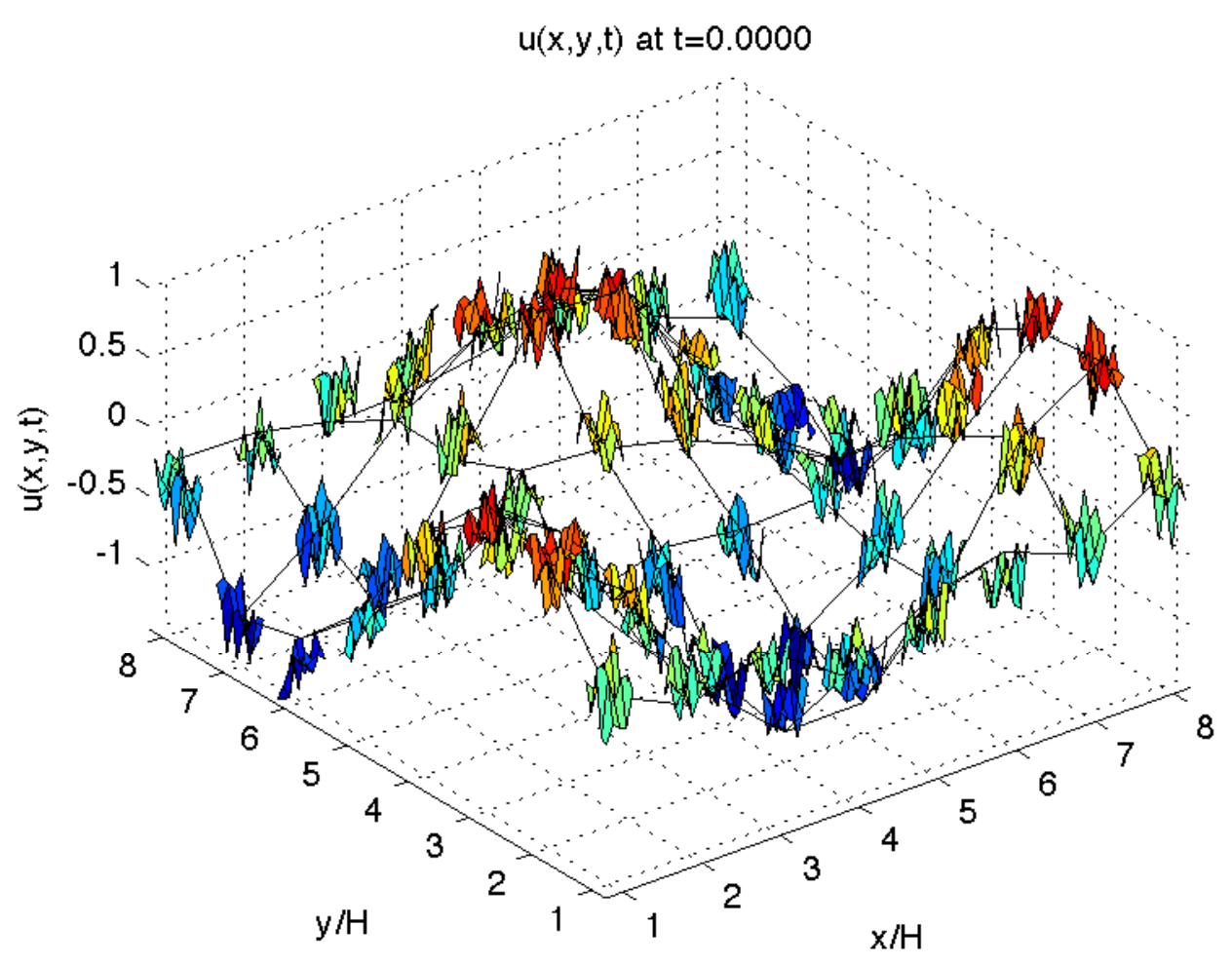}8{0}{This initial condition has random significant microscale fluctuations superimposed upon a large scale variation. This initial condition evolves to Figures~\ref{fig:Matlab/early5-4-8-12-10D/025} and~\ref{fig:Matlab/early5-4-8-12-10D/1}.}
\patchfigure{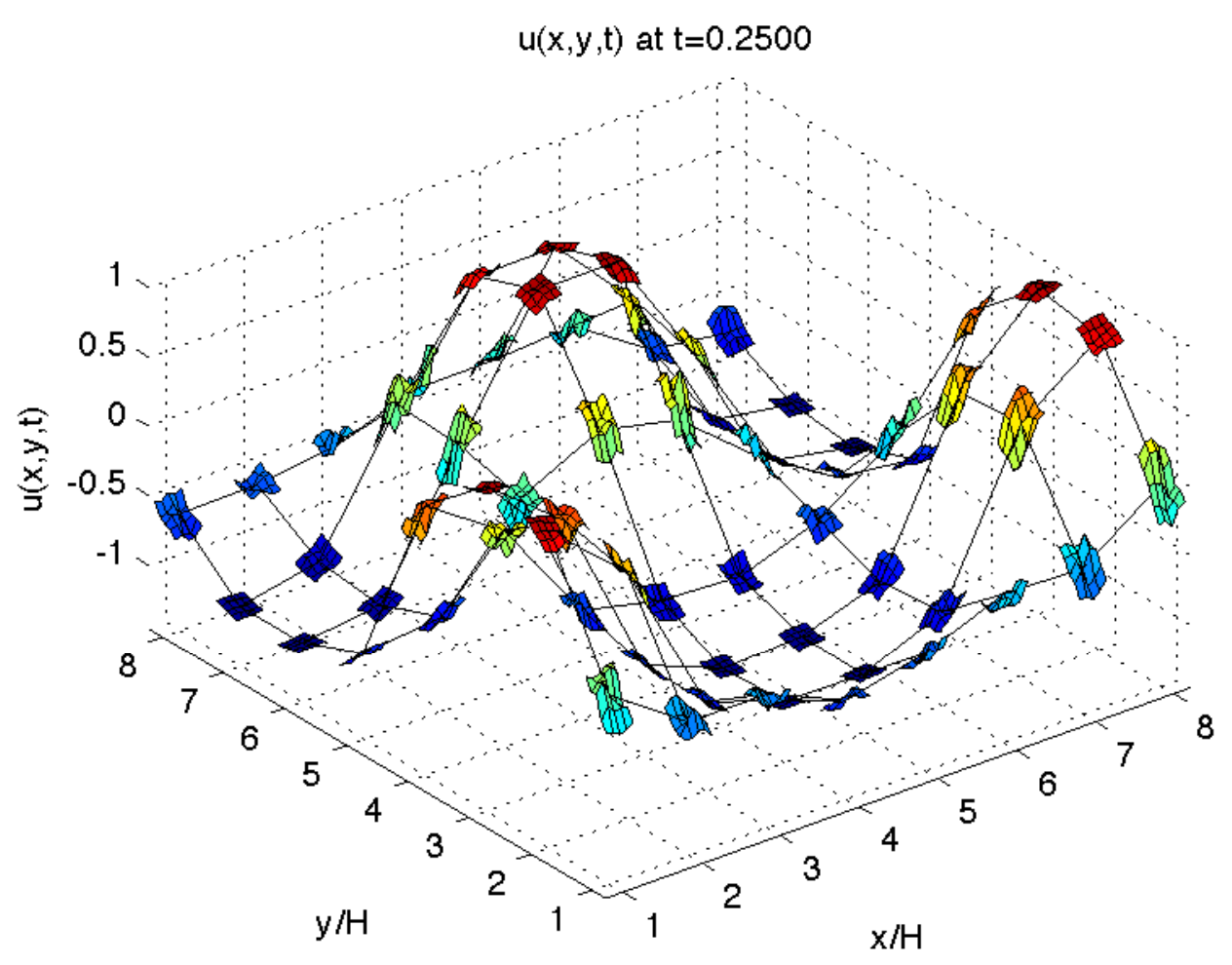}8{0.25}{At this time the spatial fluctuations within each patch have nearly smoothed to reflect the macroscale variations.}
\patchfigure{1}8{1}{By now the variations within each patch are smooth and the patch evolution reflects the dynamics of macroscale pattern formation.}

The key to support the gap-tooth scheme of Kevrekidis et al.~\cite{Kevrekidis03b} is to couple patches of spatial dynamics across space; Figures~\ref{fig:Matlab/early5-4-8-12-10D/0}--\ref{fig:Matlab/early5-4-8-12-10D/1} show an example.
Such coupling needs to preserve the accuracy and stability characteristics of the microscale dynamics.  
Here we use the dynamical systems theory of centre manifolds~\cite[e.g.]{Carr83b, Chicone97, Vanderbauwhede89} to guarantee a controlable level of fidelity between the microscale and the macroscale simulation.
Thus the first contribution of this article is to extend the dynamical systems approach of the so-called `holistic discretisation'~\cite[e.g.]{Roberts98a, MacKenzie05a} from one spatial dimension to the discrete modelling of the class of two dimensional, homogeneous, nonlinear reaction-diffusion equations 
\begin{equation}
\D tu=\nabla\cdot\big[f(u,\nabla u)\nabla u\big]+\alpha g(u),
\label{Erde}
\end{equation} 
for the field~$u(x,y,t)$.  In principle we could consider the \pde\ on any typical domain with Dirichlet, Neumann or mixed boundary conditions~\cite{Roberts01b}; however, for simplicity, in this article we generally restrict attention to spatially periodic solutions so that the modelling is homogeneous in space. 
Generalisation to spatial dimensions higher than two appears straightforward.

To achieve this aim of providing effective systematic closures for macroscale models in multiple dimensions, our approach systematically models subgrid microscale processes.  
For example, continuing the gap-tooth simulation of Figures~\ref{fig:Matlab/early5-4-8-12-10D/0}--\ref{fig:Matlab/early5-4-8-12-10D/1} would enable reasonable exploration of the competition between meta-stable macroscale domains where $u\approx \pm1$ in the Ginzberg--Landau \pde~\eqref{E_gl2d}.
Sections \ref{S_2D_divide}~and~\ref{sec:nccec} discuss two distinct avenues of theoretical support for our modelling: respectively that of centre manifold theory~\cite[e.g.]{Carr83b, Vanderbauwhede89, Kuznetsov95, Chicone97}, and that of classic consistency.  
Such dual justification is a strength of this approach.  

The complex dynamics we address arise through the interaction over space of local microscale dynamics whether in a \pde\ such as~\eqref{Erde}, or a discrete lattice equation~\cite[e.g.]{Hyman86a, Greenside96}, or a microscale simulator (Figures~\ref{fig:Matlab/early5-4-8-12-10D/0}--\ref{fig:Matlab/early5-4-8-12-10D/1}).  
Analysing the dynamics of a \pde\ for fixed macroscopic grid spacing, a third contribution of this article is to use centre manifold techniques to underpin accurate models of nonlinear dynamics by resolving naturally the dominant subgrid microscale structures and their interactions, both internally and with macroscale variations.  
Instead of imposing a subgrid field, such as the usual polynomial interpolation of finite differences and finite elements, here the \pde\ determines the subgrid field. 
Then the derived macroscale closures enable a relatively coarse numerical grid to significantly improve computational speed and stability in numerical solutions of the \pde.  
For example, we expect more extreme parameter regimes may be explored without the need for artificial hyper-viscosities~\cite[e.g.]{Zhang00}.

The analysis of \pde{}s in Sections~\ref{S_2D_divide}--\ref{chapnumcm} parallels, and has much commonality with, our analysis of dynamics on disjoint spatial patches in Sections~\ref{S_2D_divide} and~\ref{sec:patch}.
A micro-simulator within each patch requires boundary values.
If the microsimulator were to be executed over the entire macro-domain, then such boundary values come naturally from immediately neighbouring fine grid points; such neighbours are missing in gap-tooth simulation such as those of Figures~\ref{fig:Matlab/early5-4-8-12-10D/0}--\ref{fig:Matlab/early5-4-8-12-10D/1}.
Instead we propose the innovation that classic Lagrange interpolation from surrounding macroscale grid values provides the accurate coupling for the small microscale patches, analogous to accurate coupling of one dimensional dynamics~\cite{Roberts04d, Roberts06d}.
The centre manifold theory of Section~\ref{S_2D_divide} then supports the macroscale modelling.  
To complement this dynamical systems support, Theorem~\ref{thm:ipc} provides support for the consistency of the approach: the order of consistency growing linearly with the order of the interpolation.
This classic coupling rule establishes a strong connection between classic finite difference discretisations of \pde{}s, finite elements, and the methodology of the gap-tooth method.

Our work here presents two faces to computational simulation.  On the one hand we present a preprocessing methodology for generating potentially highly accurate closures of discretisations of \pde{}s or lattice dynamics.  These would closures would subsequently be used to markedly speed up simulations.  On the other hand we prove that the same approach provides coupling conditions for accurate and effective, on the fly, closures for the `equation-free' macroscale simulation of highly detailed microscale dynamics.\footnote{In either case, the issue of parallelising the computational simulations are the same, and familiar from usual approximations: to obtain higher order accuracy, generally use a wider computational stencil, which requires proportionally more communication between parallel processors in some domain decomposition of the computation.}

As a particular example, previewed in Figures~\ref{fig:Matlab/early5-4-8-12-10D/0}--\ref{fig:Matlab/early5-4-8-12-10D/1}, Sections~\ref{S_2D_low}, \ref{chapnumcm} and~\ref{sec:patch} explore in some detail the modelling of the real valued, two dimensional, Ginzburg--Landau equation obtained from the \pde~\eqref{Erde} with cubic reaction, $g=u-u^3$, and constant diffusion, $f=1$\,, namely
\begin{equation}
\D tu=\nabla^2u+\alpha(u-u^3)\,. \label{E_gl2d}
\end{equation}
We choose this 2D~real Ginzburg--Landau equation as an example prototype reaction-diffusion \pde\ because its dynamics are well understood~\cite[e.g.]{Gibbon93, Levermore96}, and because this \pde\ is important as a phenomenological model~\cite[p.6, e.g.]{Lowndes1999}: much interest lies in the long time evolution of the interacting domains of the quasi-stable states $u\approx\pm1$\,.
In Section~\ref{S_num_2D_perf} the steady states of the example2D~Ginzburg--Landau equation~\eqref{E_gl2d}, and their stability, measure the accuracy and effectiveness of various order models in this application.
Section~\ref{S_2D_perf} briefly compares a low order model with a classic finite difference model to indicate their similar performance and to provide a base for comparing high order models.
MacKenzie~\cite{MacKenzie05} reported that our modelling, based upon a mosaic of local dynamics, is much less costly to use than a global description of an inertial manifold~\cite[e.g.]{Temam90, Marion89, Jolly90}.

The macroscale model is based upon dividing the domain into either overlapping finite elements or disjoint small patches.
Following analogues in one dimension~\cite[e.g.]{Roberts00a,MacKenzie05a}, neighbouring elements\slash patches are coupled with strength parametrised by~$\gamma$.
Section~\ref{S_2Dcm} then discusses how centre manifold theory~\cite[e.g.]{Carr83b, Vanderbauwhede89, Kuznetsov95, Chicone97} assures us of the existence of an exactly closed discrete model.
Further, this discretisation is exponentially quickly attractive---it contains the emergent dynamics.
Although we cannot exactly construct this closure, theory asserts it may be approximated to any order in the strength of the inter-element\slash patch coupling~$\gamma$ and the nonlinearity~$\alpha$.

The special coupling conditions~\eqref{EbcsdL} assure us that the resultant macroscale discrete models are \emph{also} consistent with the dynamics of the reaction-diffusion \pde\ (Section~\ref{sec:nccec}).
A further contribution of this article is that the proofs for consistency in Section~\ref{sec:nccec} are new and more powerful, leading to new theorems on nonlinear and multidimensional consistency, and to a new theorem on the consistency of patch dynamics.

Section~\ref{S_2D_low} outlines the construction, consistency and predictive accuracy of low order asymptotic approximations to the macroscale discretisation of the \glpde{}.
To extract another order of accuracy from the algebra, we find (for the first time) the adjoint operator of the diffusion operator on the elements\slash patches with the nonlocal coupling conditions.
The null space of this adjoint, strikingly similar to a Galerkin basis, enables us to use an integral solvability condition to construct the third order discrete model.

However, we cannot algebraically find higher order models nor any of the patch models.  
This inability to construct algebraic approximations is one major difference between systems in one and multiple spatial dimensions.
In a further contribution, Section~\ref{chapnumcm} introduces how to numerically construct the microscale subgrid field and its evolution in 2D reaction-diffusion \pde{}s using the Ginzburg--Landau \pde\ as an example.
Such integration of numerical solutions for the subgrid field in complex algebraic expressions for the macroscale parametrisation is novel.
We find that even a relatively coarse subgrid microscale resolution is adequate to reasonably accurately underpin the macroscale modelling.

\section{Divide the domain into elements\slash patches}
\label{S_2D_divide}

We place the discrete macroscale modelling of general, two dimensional, reaction-diffusion dynamics within the purview of centre manifold theory by dividing the domain into either overlapping square elements or into small disjoint separated patches, as shown schematically in Figure~\ref{ch5f2dst}.

\begin{figure}
\begin{center}
\small \setlength{\unitlength}{0.26em}
\begin{picture}(80,80)
\thicklines 
\multiput(10,0)(20,0){4}{
  \multiput(0,0.5)(0,3){27}{\line(0,1){2}}}
\multiput(0,10)(0,20){4}{
  \multiput(0.5,0)(3,0){27}{\line(1,0){2}}}
\multiput(20,20)(20,0){3}{
\multiput(0,0)(0,20){3}{\circle*{2}} } 
\put(12,15){$i-1,j-1$}
\put(34,15){$i,j-1$} \put(52,15){$i+1,j-1$} \put(14,35){$i-1,j$}
\put(38,35){$i,j$} \put(54,35){$i+1,j$} \put(12,55){$i-1,j+1$}
\put(34,55){$i,j+1$} \put(52,55){$i+1,j+1$}
\put(30,6){\vector(1,0){20}}
\put(50,6){\vector(-1,0){20}}
\put(40,3){$H$}
\put(36,26){$2rH$}
\color{blue}
\multiput(20,20)(40,0){2}{\line(0,1){40}}
\multiput(20,20)(0,40){2}{\line(1,0){40}}
\put(40,24.5){\vector(1,0){20}}
\put(40,24.5){\vector(-1,0){20}}
\color{magenta}
\multiput(34,34)(12,0){2}{\line(0,1){12}}
\multiput(34,34)(0,12){2}{\line(1,0){12}}
\put(40,31.5){\vector(1,0){6}}
\put(40,31.5){\vector(-1,0){6}}
\end{picture}
\end{center}
\caption{Discretise a 2D domain into square elements\slash patches.
The $i,j$th~element\slash patch (solid blue\slash magenta) is centred upon the grid point $(X_i,Y_j)$: when $r=1$ (blue) $E_{i,j}$~overlaps neighbouring elements to extend to the neighbouring grid points; when $r<1/2$ (magenta) $E_{i,j}$~forms a patch separated by empty space (gaps) from neighbouring patches.}
\label{ch5f2dst}
\end{figure}
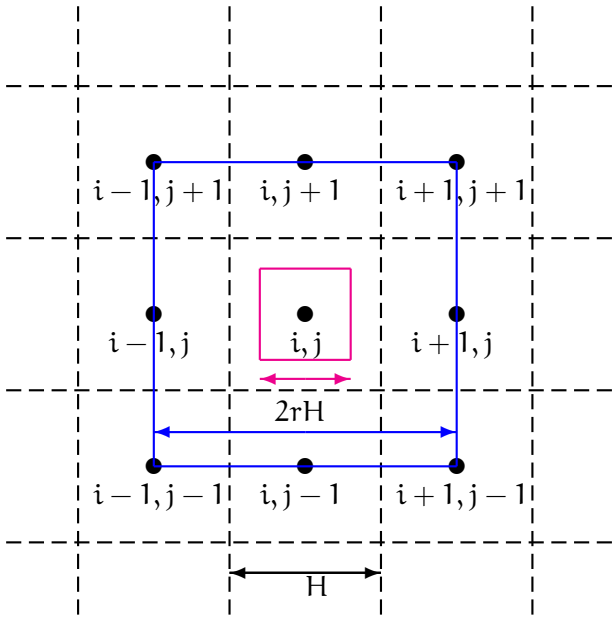

\subsection{Extend non-local coupling conditions to multiple dimensions}

Define a grid of points~$(X_i,Y_j)$ with, for simplicity, constant spacing~$H$, see Figure~\ref{ch5f2dst}.
The $i,j$th element,~$E_{i,j}$, is centred upon~$(X_i,Y_j)$ and of width $\Delta x=\Delta y = 2rH$\,: when $r=1$ we cater for the overlapping elements of holistic discretisation~\cite[e.g.]{Roberts00a}; when $r<1/2$ we cater for the spatially separated patches of equation free modelling~\cite[e.g.]{Roberts06d, Kevrekidis09a}.
Let $u_{i,j}(x,y,t)$ denote the field in the $i,j$th~element\slash patch.
The fields~$u_{i,j}(x,y,t)$ evolve according to the reaction-diffusion \pde~\eqref{Erde}; that is, \begin{equation}
\D t{u_{i,j}}=\nabla\cdot\big[f(u_{i,j},\nabla u_{i,j})\nabla u_{i,j}\big]+\alpha g(u_{i,j})\,.
\label{Erdev}
\end{equation}
The original field~$u(x,y,t)$ is then predicted by $u_{i,j}(x,y,t)$ for $(x,y)\in E_{i,j}$. 

The evolution of the field over the whole domain depends upon how the elements\slash patches are coupled together.
To couple the dynamics of each overlapping element to its neighbours, the case $r=1$\,, we use `inter-element coupling conditions' (\icc) around the $i,j$th~element of
\begin{equation}
\begin{cases}
u_{i,j}(X_{i\pm1},y,t)=
\gamma u_{i\pm1,j}(X_{i\pm1},y,t)+(1-\gamma)u_{i,j}(X_i,y,t)
,&  |y-Y_j|<H\,, \\
u_{i,j}(x,Y_{j\pm1},t)=
\gamma u_{i,j\pm1}(x,Y_{j\pm1},t)+(1-\gamma)u_{i,j}(x,Y_j,t),
& |x-X_i|<H\,. 
\end{cases}
\label{EbcsdL}
\end{equation}
These \icc\ are a natural extension to 2D of \icc\ established for 1D dynamics~\cite[e.g.]{Roberts00a}.
The crucial feature is: with $\gamma=0$ the elements are effectively isolated from each other, dividing the domain into decoupled elements with consequently independent dynamics; whereas with $\gamma=1$ these \icc\ ensure sufficient continuity between elements to recover the original problem over all space.
The \icc~\eqref{EbcsdL} embeds the physical problem, parameter $\gamma=1$\,, into a family of problems, general~$\gamma$, and then we access the physical problem from the tractable base at parameter $\gamma=0$\,.
Modelling via these overlapping elements is called `holistic' because within these elements we resolve subgrid structures by systematically approximating solutions of the \pde~\eqref{Erdev}; that is, the \pde\ itself tells us what are appropriate subgrid fields. 
In contrast, methods such as finite differences, finite elements and finite volumes, impose an assumed subgrid field upon the elements (typically a relatively low order multinomial).
Interestingly, the coupling~\eqref{EbcsdL} of overlapping elements appears to have analogues in other multiscale methods: the `border regions' of the heterogeneous multiscale method~\cite[e.g.]{E04}, the `buffers' of the gap-tooth method~\cite[e.g.]{Samaey03b}, and the overlapping domain decomposition that improves convergence in waveform relaxation of parabolic \textsc{pde}s~\cite[e.g.]{Gander98}.

We now consider the multiscale case of how to couple patches across space, as shown in Figures~\ref{fig:Matlab/early5-4-8-12-10D/0}--\ref{fig:Matlab/early5-4-8-12-10D/1}.
To couple the dynamics of each separated patch to its neighbours, across empty space as this is the case $r<1/2$\,, we use the different coupling conditions (\icc) around the $i,j$th~patch of
\begin{equation}
\begin{cases}
u_{i,j}(X_i\pm rH,y,t)= 
\shift _i^{\pm r}(\gamma)\shift _j^{\pm\eta}(\gamma)u_{i,j}(X_i,Y_j,t)\,,
&  |y-Y_j|<rH\,, \\
u_{i,j}(x,Y_j\pm rH,t)= 
\shift _i^{\pm \xi}(\gamma)\shift _j^{\pm r}(\gamma)u_{i,j}(X_i,Y_j,t)\,,
& |x-X_i|<rH\,, 
\end{cases}
\label{eq:epcc}
\end{equation}
in terms of subgrid variables $\xi=(x-X_i)/H$ and $\eta=(y-Y_j)/H$, and the discrete shift operator~$\shift $.  
The reason for the difference between the \icc~\eqref{EbcsdL} and~\eqref{eq:epcc} arises because the patch \icc~\eqref{eq:epcc} depend only upon the neighbouring grid values $\uij(t)=u_{i,j}(X_i,Y_j,t)$: in large scale computational simulation, the \icc~\eqref{eq:epcc} minimise communication between patches in comparison with the \icc~\eqref{EbcsdL} which require data along the mid-patch lines to be communicated with neighbouring elements.
The shift operators $\shift _i^{\pm \xi}(\gamma)$~and~$\shift _j^{\pm \eta}(\gamma)$ on the right-hand side of the \icc~\eqref{eq:epcc} derive from classic Lagrange interpolation through neighbouring grid values, but ameliorated by the coupling parameter~$\gamma$.  These $\gamma$~ameliorated shift operators are best defined in terms of the well known centred difference~$\delta$ and mean~$\mu$ operators: for each direction indicated in~\eqref{eq:epcc} by $i$~and~$j$, the $\gamma$~ameliorated shift by an fraction~$\pm\xi$ of a grid spacing is represented in classic operator algebra~\cite[e.g.]{npl61, Hildebrand1987} as
\begin{align}
\shift ^{\pm\xi}(\gamma)&=\left[1+\gamma(\pm\mu\delta+\rat12\delta^2)\right]^\xi
\nonumber\\&
=1+\xi\gamma(\pm\mu\delta+\rat12\delta^2)
+\rat12\xi(\xi-1)\gamma^2(\pm\mu\delta+\rat12\delta^2)^2
+\cdots\,.
\label{eq:shifty}
\end{align}
Identically for the \icc~\eqref{EbcsdL}, a crucial feature of the \icc~\eqref{eq:epcc} is: with $\gamma=0$ the patches are effectively isolated from each other, dividing the domain into decoupled patches with consequently independent dynamics; whereas with $\gamma=1$ these \icc\ ensure each patch matches those in its neighbourhood via classic Lagrange interpolation.

Whether patches or overlapping elements we need domain boundary conditions to form a well posed problem.
For definiteness in theoretical support, let there be $m$~elements\slash patches in the spatial domain~$\Omega$ with the field required to be periodic in both $x$~and~$y$, and the field to be in a Sobolev space such as~$H^2(\Omega)$.
For example, the elements\slash patches may form a $\sqrt m\times\sqrt m$ grid in the domain (any factorisation of~$m$ is feasible).
In principle, the 2D~elements\slash patches could be any shapes, regular or irregular; square ones appear to be easiest to start with.
Although physical domain boundary conditions have been explored for $1$D domains~\cite[e.g.]{Roberts01b, MacKenzie03}, in this initial work we avoid issues of \emph{physical domain} boundary conditions on the domain~$\Omega$ by the doubly periodic condition.
Such doubly periodic boundary conditions enable the most straightforward theoretical statements of support.

\subsection{Centre manifold theory supports discrete models}
\label{S_2Dcm}
This section describes in detail how the \icc{}~\eqref{EbcsdL} or~\eqref{eq:epcc} lead to centre manifold theory~\cite[e.g.]{Carr83b, Vanderbauwhede89, Kuznetsov95, Chicone97} supporting an accurately closed, discrete model for general reaction-diffusion systems~\eqref{Erde}  via its dynamics~\eqref{Erdev} within coupled elements.  Figure~\ref{fig:patcheigenv} shows one example of the scale separation that underlies this modelling: the figure shows a large spectral gap between the `small' eigenvalues (decay rates) of the macroscale modes and the eigenvalues${}<-40$ of the rapidly decaying microscale modes. 

A homotopy in the coupling parameter~$\gamma$ connects the physically relevant macroscale discretisation to a theoretically tractable base.
When parameters $\alpha=\gamma=0$ both the \pde\ reaction and the coupling on the right-hand side of the \icc~\eqref{EbcsdL} or~\eqref{eq:epcc} disappear.
The elements\slash patches are then effectively isolated from each other and so the resultant diffusion in the \pde~\eqref{Erdev} is particularly simple: exponentially quickly in time, the field~$u_{i,j}$ becomes independently constant within each element.
We use this $m$~parameter family of piecewise constant solutions as a basis for analysing the case when the elements\slash patches are coupled together, $\gamma\neq0$\,.
Particularly interesting is the approximation for full coupling, $\gamma=1$\,, when the \pde~\eqref{Erde} is effectively restored: over the whole domain because \icc~\eqref{EbcsdL} then ensure sufficient continuity between adjacent elements as described previously for \pde{}s in one spatial dimension~\cite[e.g.]{Roberts98a, MacKenzie05a, Roberts06d}; or because the \icc~\eqref{eq:epcc} connect patches over the gaps via clasic Lagrange interpolation.

The support of centre manifold theory~\cite[e.g.]{Carr83b, Vanderbauwhede89, Kuznetsov95, Chicone97} is based upon a linear picture of the dynamics.
Adjoin to the \pde~\eqref{Erde}  the dynamically trivial equations
\begin{equation}
    \frac{\partial \alpha}{\partial t}=\frac{\partial \gamma}{\partial t}=0\,,
    \label{EtrivL}
\end{equation}
and consider the reaction-diffusion dynamics in the extended state space $(u_{i,j}(x,y),\gamma,\alpha)$.
In this extended state space, points $\alpha=\gamma=0$ and $u_{i,j}={}$constant, say~$\uij$, are equilibria of the diffusion~\eqref{Erdev}, hence these form a subspace of equilibria in the extended state space, the subspace $\bbE_0=\{(\uij,0,0)\mid \text{for all }U_{i,j}\}$.
Linearized about each of the equilibria in~$\bbE_0$, the \pde\ for perturbations~$u'_{i,j}(x,y,t)$ within each element\slash patch is then the constant coefficient diffusion \pde\ 
\begin{equation}
    \frac{\partial u'_{i,j}}{\partial t}=f_{i,j}\nabla^2u'_{i,j}
    \quad\text{for }(x,y)\in E_{i,j}\text{ for each }i,j,
    \label{eq:ccdpde}
\end{equation}
where the constant diffusivities $f_{i,j}=f(\uij,0)$.
These \pde{}s are decoupled because they are to be solved with the $\gamma=0$ \icc
\begin{equation}
\begin{cases}
u'_{i,j}(X_{i}\pm rH,y,t)=u'_{i,j}(X_i,\{y,Y_j\},t),& |y-Y_j|<rH\,,
\\
u'_{i,j}(x,Y_{j}\pm rH,t)=u'_{i,j}(\{x,X_i\},Y_j,t),& |x-X_i|<rH\,,
\end{cases}
\label{eq:ccdbc}
\end{equation}
where the first alternative in braces on the two right-hand sides corresponds to the elemental \icc~\eqref{EbcsdL}, whereas the second alternative in braces corresponds to the patch \icc~\eqref{eq:epcc}.
Among other eigenmodes such as $u'_{i,j}=1$ with eigenvalue zero, separation of variables shows that the following are some of the linear eigenmodes associated with each element\slash patch:
\begin{equation*}
\alpha=\gamma=0\,,
\quad
u'_{i,j}\propto e^{\lambda_{i,j,k,l} t} \sin(k\pi\xi/r) \sin(l\pi\eta/r) ,
\end{equation*}
inside the $i,j$th~element for all integers $k,l=1,2,3,\ldots$, where the eigenvalues corresponding to each of these modes are
\begin{equation}
    \lambda_{i,j,k,l}=-f_{i,j}\frac{(k^2+l^2)\pi^2}{r^2H^2}\,.
    \label{Eeigen}
\end{equation}
Simple numerical solutions of the \pde~\eqref{eq:ccdpde} and boundary conditions~\eqref{eq:ccdbc}\footnote{The eigenproblem for the \pde\ and boundary conditions were solved via second order finite differences on microscale grids of size $9\times9$, $17\times 17$ and $33\times33$.
Then the eigenvalues were extrapolated with a Shanks transform to approximately  the reported accuracy.} confirm that~\eqref{Eeigen} for $k,l=0,1,2,3,\ldots$ are the only eigenvalues on elements ($k,l=0$ included here), but that on patches the eigenvalues, in addition to~\eqref{Eeigen} for $k,l=1,2,3,\ldots$, are~$\hat\lambda f_{i,j}\pi^2/(r^2H^2)$ for $\hat\lambda=0,-1.250,-3.250,-3.73\pm i2.00,\ldots$\,.   
These eigenvalues~$\hat\lambda f_{i,j}\pi^2/(r^2H^2)$ of the decoupled dynamics approximate the clusters shown by the histogram of Figure~\ref{fig:patcheigenv} for the fully coupled dynamics.
In addition there are the two trivial `parameter' eigenmodes: firstly $\gamma={}$constant and $\alpha=u'_{i,j}=0$\,; and secondly, $\alpha={}$constant and $\gamma=u'_{i,j}=0$\,.
In a spatial domain with $m$~elements\slash patches and when all diffusivities~$f_{i,j}>0$\,, then all eigenvalues are negative, $-f_{i,j}\pi^2/(r^2H^2)$ or less, except for $m+2$ zero eigenvalues.
Of the $m+2$ zero eigenvalues, one is associated with each of the $m$~elements\slash patches and two come from the trivial equations~\eqref{EtrivL} for the parameters.
That is, the slow subspace is $\{(u_{i,j},\gamma,\alpha)\}$ for constant~$u_{i,j}$.
The above spectrum establishes the following corollary of a centre manifold existence theorem~\cite[p.281, p.96 respectively]{Carr83b, Vanderbauwhede89}.

\begin{corollary}[Existence] \label{cor:exist}
Provided the nonlinear diffusivity~$f$ and reaction~$g$ in~\eqref{Erdev} are sufficiently smooth, and all $f_{i,j}>0$ then  an $m+2$~dimensional slow manifold~${\cal M}$ exists for~\eqref{Erdev}~and~\eqref{EbcsdL}, or \eqref{Erdev}~and~\eqref{eq:epcc}, in some finite neighbourhood of the subspace~$\bbE_0$ of equilibria.\footnote{Keep clear the distinction between centre manifold theory and the slow manifolds discussed here: centre manifold theory applies to systems where the \emph{real part} of the eigenvalues of critical modes are zero; whereas here we explore and construct the particular case of slow manifolds because here the relevant eigenvalues are \emph{precisely} zero.}
\end{corollary} 
The slow manifold ${\cal M}$ is parametrized both by the two parameters $\gamma$~and~$\alpha$, and by a measure of the field in each element\slash patch; we use the grid value $\uij(t)=u_{i,j}(X_i,Y_j,t)$ to measure the field in the $i,j$th~element\slash patch.
Using $\vec U$ to denote the vector of such parameters, we write the slow manifold~${\cal M}$  as the subgrid fields
\begin{equation}
    u_{i,j}=u_{i,j}(x,y;\vec U,\gamma,\alpha).
    \label{EcmvL}
\end{equation}
These functions~$u_{i,j}(x,y;\vec U,\gamma,\alpha)$, that Sections \ref{S_2D_low}~and~\ref{chapnumcm} construct for the \glpde{}, resolve the subgrid scale physical structures as a function of the neighbouring grid values in~$\vec U$.
On this slow manifold~${\cal M}$ the grid values~$\uij$ evolve deterministically:
\begin{equation}
    d \uij/dt=\dotuij=g_{i,j}(\vec U,\gamma,\alpha)\,,
    \label{EcmgL}
\end{equation}
where $g_{i,j}$~is the restriction of~\eqref{Erdev}~and~\eqref{EbcsdL}, or \eqref{Erdev}~and~\eqref{eq:epcc}, to the slow manifold~${\cal M}$.
In essence, this closure of the grid scale dynamics comes from the resolution of subgrid scale structures.
The set of \ode{}s~\eqref{EcmgL} form the discrete model of the original \pde.

Using the value of the field  at the grid points to parametrise the slow manifold provides the necessary `amplitude conditions' to close the problem:
\begin{equation}
\uij=u_{i,j}(X_i,Y_j;\vec U,\gamma,\alpha).
\label{Eampl}
\end{equation}  
Many other amplitude conditions are possible such as defining the `amplitudes'~$\uij$ to be the mean field over the $i,j$th~element\slash patch.
However, using the grid values are simple, traditional, and have a direct physical meaning.

Centre manifold theorems~\cite[e.g.]{Carr83b, Vanderbauwhede89} also support the following crucial emergence (called asymptotic completeness by Robinson~\cite{Robinson96}) and approximation properties.

\begin{corollary}[Emergence and approximation] \label{cor:rc}
Provided the nonlinear diffusivity~$f$ and reaction~$g$
in~\eqref{Erdev} are sufficiently smooth, and all $f_{i,j}>0$\,, then  
\begin{itemize}
    \item every solution of the nonlinear reaction-diffusion dynamics~\eqref{Erdev}~and~\eqref{EbcsdL}, or \eqref{Erdev}~and~\eqref{eq:epcc}, that stays within a neighbourhood of the slow manifold~$\cal M$,~\eqref{EcmvL}, approaches exponentially quickly a solution of the discrete model~\eqref{EcmgL} on the slow manifold~\eqref{EcmvL}; and

    \item the order of error in asymptotically approximating the slow manifold and its evolution, \eqref{EcmvL}--\eqref{EcmgL}, is the same as the order of residuals of the governing equations~\eqref{Erdev}~and~\eqref{EbcsdL}, or \eqref{Erdev}~and~\eqref{eq:epcc}.
In particular, because the base equilibria form a subspace~$\bbE_0$, here the approximation is global in the grid values~$\uij$, it is \emph{local only} in the coupling~$\gamma$ and nonlinearity~$\alpha$.
\end{itemize}
\end{corollary}

\section{A slow manifold discretisation of the Ginzburg--Landau PDE}
\label{S_2D_low}

We now explore constructing slow manifold, discrete models for a specific example reaction-diffusion system, the 2D \glpde.
In applications this construction is a preprocessing step prior to large scale numerical simulation.  The support of centre manifold theory for our exotic closures suggests subsequent numerical simulation will be significantly more accurate and/or efficient.

This section only addresses the slow manifold discretisation on overlapping elements with the aim of deriving and testing accurate nonlinear closures of the reaction-diffusion dynamics; later sections return to the slow manifold view of the dynamics on spatially separated patches.
In contrast to one spatial dimension where it is straightforward to construct slow manifold discretisations for general \pde{}s (see our web service~\cite{Roberts02a,Roberts02b}), one obstacle in multiple dimensions is the limited range of slow manifolds representable by multi-variate polynomials.
Consequently, the next section extends the approach via numerical solution of the subgrid fields.
The subgrid fields on patches require such numerical solutions and so patches are left for later sections. 

Substituting the slow manifold ansatz \eqref{EcmvL}~and~\eqref{EcmgL} into the  \glpde, assuming solutions are doubly periodic, we obtain the {\textsc{pde}} to solve for the model by equating the \pde\ for~$\D t{u_{i,j}}$ to that obtained by the chain rule:
\begin{equation}
\D t{u_{i,j}}=\sum_{k,l}\D{U_{k,l}}{u_{i,j}} g_{k,l}
=\nabla^2u_{i,j}+\alpha \left(u_{i,j}-u_{i,j}^3\right)\,. \label{EcmpdeL}
\end{equation}
To construct the slow manifold~\eqref{EcmvL}--\eqref{EcmgL} by solving the \pde~\eqref{EcmpdeL} with coupling and amplitude conditions involves considerable algebraic detail.
Computer algebra handles the algebraic details of the construction by iteration~\cite[e.g.]{Roberts96a}.
The precise procedure used here is fully documented on a freely accessible, preprint server~\cite{MacKenzie11a}
to empower readers to check, reproduce and build upon our quoted results.
The methodology constructs a model by driving to zero, to some order of error, the residuals of the governing differential equation~\eqref{EcmpdeL} and the inter-element\slash patch coupling \icc~\eqref{EbcsdL} or~\eqref{eq:epcc}.
By the Approximation Corollary~\ref{cor:rc} we construct correspondingly accurate approximations to the slow manifold of~\eqref{EcmpdeL}.
These approximations, upon setting coupling parameter $\gamma=1$ and the nonlinearity parameter~$\alpha$ appropriately, form 2D discrete models of the 2D~\glpde{}.

Recall that although the approximations to the slow manifolds considered here are asymptotic, the existence of an emergent slow manifold holds in a \emph{finite} domain around the slow subspace~$\bbE_0$ (Corollaries~\ref{cor:exist}--\ref{cor:rc}).   Thus although we \emph{construct} and label the slow manifold approximations via asymptotics in small parameters, the resulting models hold for \emph{finite} parameter values.  The issue is how well do the approximations perform: we give evidence that reasonable results are straightforwardly obtained for low order approximations for nonlinear parameter~$\alpha$ up to about~$30$.

One consequence of using computer algebra is that there is no need to record in this article most of the considerable algebraic detail in constructing the models.
Those wishing to verify the correctness of the results recorded herein should download and examine our freely available technical report~\cite{MacKenzie11a} that details the precise computer algebra procedure.
Because the algorithm is based upon driving the residuals to zero, the critical aspect of the procedure is simply the correct coding of the computation of the residuals of the governing equations.
One may straightforwardly edit the code~\cite{MacKenzie11a} to construct holistic discretisations for other reaction-diffusion systems in the class~\eqref{Erde}.

\paragraph{The $\Ord{\gamma^3+\alpha^3}$ holistic discretisation}
\label{S_2d_hol}
Satisfying the \pde\ and \icc{} to residuals of~$\Ord{\gamma^3+\alpha^3}$ the computer algebra procedure~\cite[\S2.2]{MacKenzie11a} gives subgrid fields which are too complex to record here.
The corresponding evolution of the grid values on the slow manifold are
\begin{align}
\dotuij=&\frac{\gamma}{H^2}\bdelta^2 \uij +\alpha\left(\uij - \uij^3\right)
-\frac{\gamma^2}{12H^2}\bdelta^4 \uij
\nonumber\\&{}
+\alpha\gamma \frac1{12}\left\{
\left[3(\mu_i\delta_i U_{i,j})^2+\rat14(\delta_i^2 U_{i,j})^2\right](2+\delta_i^2)
+(\delta_i^2 U_{i,j})^2
\right.\nonumber\\&\qquad\left.{}
+\left[3(\mu_j\delta_j U_{i,j})^2+\rat14(\delta_j^2 U_{i,j})^2\right](2+\delta_j^2)
+(\delta_j^2 U_{i,j})^2
\right\} U_{i,j}
\nonumber\\&{}
+\Ord{\gamma^3+\alpha^3}\,,
\label{E_gl2d_g1a1}
\end{align}
where the operator difference $\delta=\shift ^{1/2}-\shift ^{-1/2}$ and mean $\mu=(\shift ^{1/2}+\shift ^{-1/2})/2$ in each of the two $i$~and~$j$ directions, and
where the bold centred difference operator applies in both spatial dimensions:
\begin{eqnarray}
\bdelta^2\uij&=&U_{i+1,j}+U_{i-1,j}+U_{i,j+1}+U_{i,j-1}-4\uij\,,
\label{eq:bdelta}\\
\bdelta^4\uij&=&U_{i+2,j}+U_{i-2,j}+U_{i,j+2}+U_{i,j-2}
\nonumber\\&&{}
-4(U_{i+1,j}+U_{i-1,j}+U_{i,j+1}+U_{i,j-1}) +12\uij\,.
\end{eqnarray}
The model~\eqref{E_gl2d_g1a1} is simply the extension to two spatial dimensions of the $\Ord{\gamma^3+\alpha^3}$ holistic model of the 1D Ginzburg--Landau \pde~\cite{MacKenzie05}.

The holistic discrete model has the dual justification of consistency with the \pde\ in addition to the justification provided by centre manifold theory.
As proven in Section~\ref{sec:nccec}, consistency for such discrete models follows from the coupling \icc~\eqref{EbcsdL}~\cite{Roberts00a}.
Set the coupling parameter $\gamma=1$ in the discrete equation~\eqref{E_gl2d_g1a1} to recover the holistic discrete model of the \glpde{} in 2D.
To test consistency, we expand the finite differences of~\eqref{E_gl2d_g1a1} in a Taylor series in the grid spacing~$H$~\cite[\S3.2]{MacKenzie11a} to find the equivalent continuum \pde{} for the $\Ord{\gamma^3+\alpha^3}$ holistic model~\eqref{E_gl2d_g1a1} is
\begin{equation}
\D tu=\nabla^2u+\alpha(u-u^3)
+\frac{H^2\alpha}{2}u|\nabla u|^2
- \frac{H^4}{90} \left( \Dn x6u+\Dn y6u\right)
+\Ord{\alpha^3+H^6}\,.
\label{E_gl2d_epde}
\end{equation}
The $\Ord{\gamma^3+\alpha^3}$ holistic model is $\Ord{H^4+\alpha^2}$ consistent, maintaining in 2D the dual justification of holistic discretisation found for 1D \pde{}s~\cite{Roberts98a, Roberts00a}.
Section~\ref{sec:nccec} proves this consistency in some generality for 2D \pde{}s~\eqref{Erde}.

\subsection{Illustration of the subgrid field in 2D}
\label{S_2D_perf}
Here we plot the subgrid fields for a coarse grid solution of the $\Ord{\gamma^2,\alpha^2}$~holistic model (obtained from~\eqref{E_gl2d_g1a1} by omitting the $\gamma^2$~term and then evaluating at $\gamma=1$) of the 2D \glpde.
We restrict attention to doubly  odd symmetric solutions that are also doubly $2\pi$-periodic:  that is,
\begin{equation}
 u(x,y,t)=\begin{cases}
-u(2\pi-x,y,t)=-u(x,2\pi-y,t),
&\text{odd symmetry,}
\\ u(x+2\pi,y,t)=u(x,y+2\pi,t),
&2\pi\text{-periodicity.}
\end{cases}
\label{E_gl_bc_po_2d}
\end{equation}
Figure~\ref{ch5fig3} shows the subgrid fields for the $\Ord{\gamma^2,\alpha^2}$~holistic model with $4\times4$ elements on $[0,\pi]\times[0,\pi]$ at nonlinearity $\alpha=6$.
The subgrid fields exhibit the nonlinear subgrid structure of the 2D~\glpde\ and its interaction through the \icc{}.
The subgrid fields are comprised of actual solutions, albeit approximate, of the 2D~\glpde.
\begin{figure}
 \centering
\includegraphics[width=\textwidth]{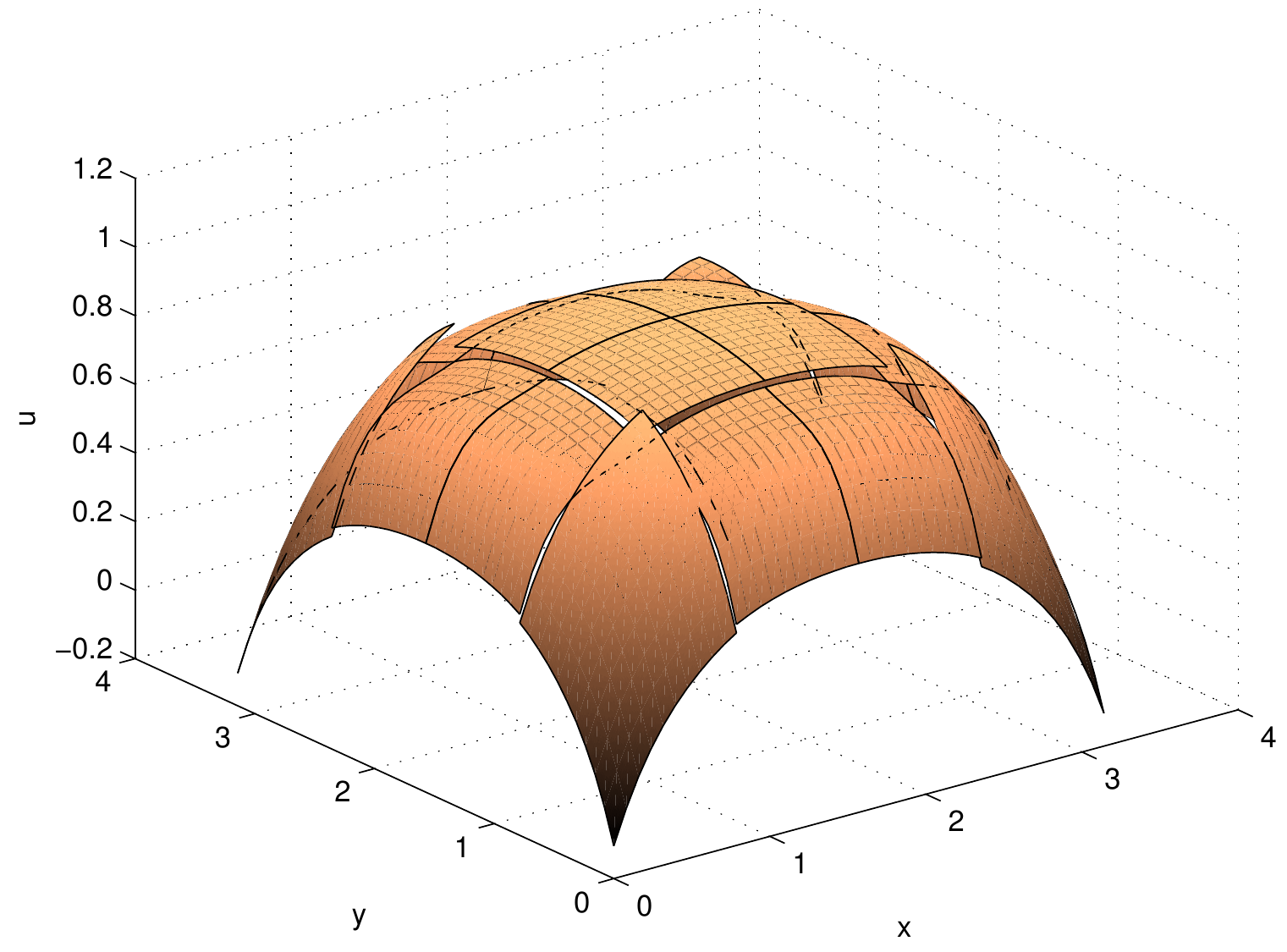}
\caption{an example of the subgrid field for the $\Ord{\gamma^2,\alpha^2}$~holistic
model with $4\times4$ elements on $[0,\pi]\times[0,\pi]$ at nonlinearity $\alpha=6$}
    \label{ch5fig3}
\end{figure}

Note the subgrid fields have small but noticeable jumps across the boundaries of the elements.
Higher order holistic models reduce these jumps across the boundaries as seen, for example, in the holistic models of the 1D \KS{} equation~\cite{MacKenzie05a}.

\paragraph{The $\Ord{\gamma^2,\alpha^2}$ holistic model needs improving}

We investigate the performance of the $\Ord{\gamma^2,\alpha^2}$ holistic model on coarse grids by comparing its bifurcation diagram to an accurate solution.
Again we restrict our attention to doubly odd symmetric solutions that are $2\pi$-doubly periodic~\eqref{E_gl_bc_po_2d}.

Continuation software \textsc{auto}~\cite{auto} and \textsc{xppaut}~\cite{xppaut} calculates this bifurcation information, as outlined  for the \KS{} equation in MacKenzie's PhD dissertation~\cite{MacKenzie05}.
In such bifurcation diagrams the blue curves indicate stable steady state solutions and red curves indicate unstable steady state solutions.
The open squares indicate steady state bifurcations.


Underlying Figure~\ref{ginz2d_num_8_bif_2}, in grey for reference, is an accurate bifurcation diagram of the 2D \glpde.
It is constructed with a computationally expensive, fourth order, centered difference, approximation with $24\times24$ points on $[0,\pi]\times[0,\pi]$.
The trivial solution undergoes steady state bifurcations at $\alpha=2,8,18$ leading to the unimodal, bimodal and trimodal branches respectively.
For nonlinearity ranging over $1<\alpha<30$, only the unimodal branch is stable; all other branches are unstable.

\begin{figure}
\centering
\includegraphics[width=\textwidth]{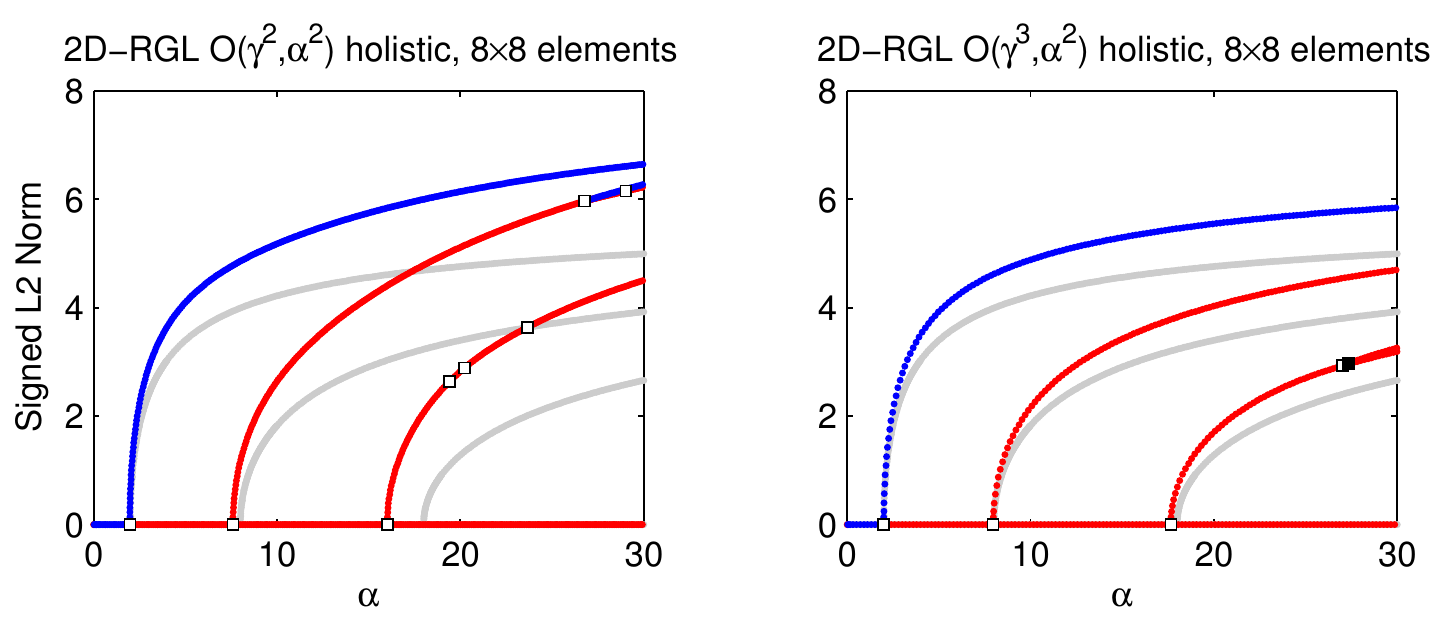}
\caption{Bifurcation diagrams for the 2D \glpde\ with $8\times8$ elements on $[0,\pi]\times[0,\pi]$ for holistic models (a)~${\cal O}(\gamma^2,\alpha^2)$  (b)~$\Ord{\gamma^3,\alpha^2}$.
The accurate bifurcation diagram is shown in grey.}
\label{ginz2d_num_8_bif_2}
\end{figure}

Figure~\ref{ginz2d_num_8_bif_2}(left) compares the bifurcation predictions of the ${\cal O}(\gamma^2,\alpha^2)$~holistic model with the exact bifurcation diagram (grey).  The predictions are reasonable for parameter~$\alpha$ up to about ten, and for the smaller range of amplitudes shown.  Figure~\ref{ginz2d_num_8_bif_2}(right) shows the significantly improved accuracy of the $\Ord{\gamma^3,\alpha^2}$~holistic model that is obtained by resolving more inter-element interactions through retaining higher orders in the coupling parameter~$\gamma$.  As proven for other \pde{}s in 1D~\cite{Roberts98a, MacKenzie03, MacKenzie05a}, higher order modelling in 2D evidently improves predictions.

%

\paragraph{Higher order models need numerical construction} 
To improve the accuracy of such discrete closures we need to compute higher orders in either coupling~$\gamma$, or nonlinearity~$\alpha$, preferably both.
Improved accuracy occurs at higher order in comparable 1D problems~\cite{MacKenzie05a}.
However, apparently it is not possible to \emph{analytically} construct higher order subgrid fields in 2D: the subgrid fields required for our closures appear to be no longer in the class of multivariate polynomials.
For example, the $\Ord{\gamma^3,\alpha^2}$~model used for Figure~\ref{ginz2d_num_8_bif_2}(right) is not immediately obtainable from the computer algebra iteration~\cite{MacKenzie11a}.
Instead, numerical methods must be used to find the subgrid fields as described in Section~\ref{chapnumcm}.
Section~\ref{sec:patch} employs the same numerical methods to also construct slow manifold models on patches.
However, the well known `solvability condition' in asymptotic mathematical methods empowers us to derive the next order in the evolution, analytically from the residuals, without needing to find the next order of the subgrid fields, and hence provides the model underlying Figure~\ref{ginz2d_num_8_bif_2}(right).

\subsection{The adjoint provides an extra order of accuracy}
\label{sec:apeoa}

We scrounge an extra order of accuracy from the `solvability condition'~\cite[e.g.]{Pavliotis07} applied to residuals of the next asymptotic order.
Because the linear operator used to find corrections to the subgrid field is singular---the operator necessarily has homogeneous solutions that compose the slow subspace~$\mathbb E_0$---the Fredholm alternative is that the `right-hand side' of the equation for the subgrid fields must lie in the range of the singular operator.
This solvability condition is enough to determine an extra correction to the evolution.  

Recall from elementary linear algebra that to be in the range of the operator, the solvability condition is that the right-hand side must be orthogonal to the null space of the adjoint operator.
Thus the first task of this section is to find the adjoint operator of the linear constant diffusion \pde~\eqref{eq:ccdpde} with its insulating boundary conditions~\eqref{eq:ccdbc}.
These non-obvious adjoints have never been identified before.
Second, we find a basis for the null space.
Lastly, computer algebra computes the required extra order in the evolution.

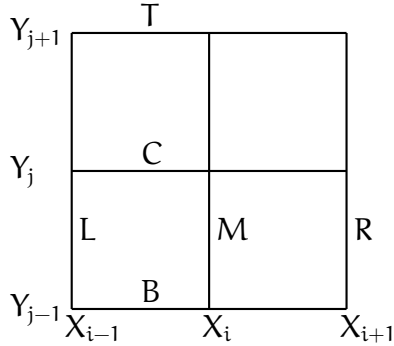
\begin{figure}
\centering
\setlength{\unitlength}{1ex}
\begin{picture}(27,25)(0,0)
\thicklines
\put(-1,-2){
\multiput(5,5)(10,0){3}{\line(0,1){20}}
\multiput(5,5)(0,10){3}{\line(1,0){20}}
\put(5.5,10){\put(0,0){$L$}\put(10,0){$M$}\put(20,0){$R$}}
\put(10,5.5){\put(0,0){$B$}\put(0,10){$C$}\put(0,20){$T$}}
\put(4.5,3){\put(0,0){$X_{i-1}$}\put(10,0){$X_i$}\put(20,0){$X_{i+1}$}}
\put(0.5,4.5){\put(0,0){$Y_{j-1}$}\put(0,10){$Y_j$}\put(0,20){$Y_{j+1}$}}
}
\end{picture}
\caption{each element is effectively divided into four subregions by the non-locality of the boundary conditions~\eqref{eq:ccdbc}.
To derive the adjoint, label the edges of these four subregions as shown.}
\label{fig:radjoint}
\end{figure}

The decoupling of the elements, provided by $\gamma=0$ in the boundary conditions~\eqref{eq:ccdbc}, simplifies finding the adjoint: we need only consider each element in isolation.
Thus define the inner product to be the integral over the $i,j$th~element:
\begin{displaymath}
\left<u,v\right>=\iint_{E_{i,j}}uv\,dx\,dy
=\int_{Y_{j-1}}^{Y_{j+1}} \int_{X_{i-1}}^{X_{i+1}} uv\,dx\,dy \,.
\end{displaymath}
To find the adjoint recognise that each element is effectively subdivided into four subregions by the coupling of the boundary values to internal values by the boundary conditions~\eqref{eq:ccdbc}, Figure~\ref{fig:radjoint} schematically shows these four regions.
In addition, there exists previously implicit conditions that the subgrid field~$u$ and its gradient are continuous throughout the element.
Then integration by parts, or the divergence theorem, transforms the inner product
\begin{align*}
\left<\delsq u,v\right>&
=\left<u,\delsq v\right>
-\int_L u_xv-uv_x\,dy  +\int_R u_xv-uv_x\,dy
\\&\qquad{} 
+\int_{M-} u_xv-uv_x\,dy -\int_{M+} u_xv-uv_x\,dy
\\&\qquad{}+\text{analogous integrals on $B$, $C$~and~$T$,}
\end{align*}
where specific parts of the boundary integrals are labelled as shown in Figure~\ref{fig:radjoint}.
Using superscripts to denote evaluation, continuity requires $u^{M\pm}=u^M$ and $u_x^{M\pm}=u_x^M$\,, and the boundary conditions~\eqref{eq:ccdbc} imply $u^L=u^M=u^R$\,.
Thus the inner product
\begin{align*}
\left<\delsq u,v\right>
&=\left<u,\delsq v\right>
+\int_{Y_{j-1}}^{Y_{j+1}} \Big[-u_x^Lv^L+u^Mv_x^L  +u_x^Rv^R-u^Mv_x^R
\\&\qquad{} 
+u_x^Mv^{M-} -u^Mv_x^{M-} -u_x^Mv^{M+} +u^Mv_x^{M+}
\Big] dy
\\&\qquad{}+\text{(analogous $x$~integrals of $y$~derivatives)}
\\&=\left<u,\delsq v\right>
+\int_{Y_{j-1}}^{Y_{j+1}} \Big[-u_x^Lv^L +u_x^Rv^R 
+u_x^M\left(v^{M-}-v^{M+}\right)
\\&\qquad{} 
+u^M\left(v_x^L-v_x^R -v_x^{M-} +v_x^{M+}\right)
\Big] dy
\\&\qquad{}+\text{(analogous $x$~integrals of $y$~derivatives).}
\end{align*}
For the adjoint, these integrals on the right-hand side must vanish for all smooth fields~$u$.
Consequently, the null space of the adjoint operator satisfies Laplace's equation~$\delsq v=0$ with conditions: firstly, that $v$~is zero around the edges~$L$, $R$, $B$~and~$T$ of the element; secondly, that $v$~is continuous on the interior partitions $M$~and~$C$ (but its gradients may be discontinuous there); thirdly, that $v_x^L -v_x^{M-} +v_x^{M+}-v_x^R=0$\,; and lastly, that $v_y^B -v_y^{C-} +v_y^{C+}-v_y^T=0$\,.
Because of these conditions, the null space of the adjoint is spanned by the `pyramid' $v=(1-|x-X_i|/H)(1-|y-Y_i|/H)$ as displayed in Figure~\ref{fig:zadjoint}.

\begin{figure}
\centering
\setlength{\unitlength}{1mm}
\begin{picture}(110,78)(0,0)
\put(2,0){\includegraphics{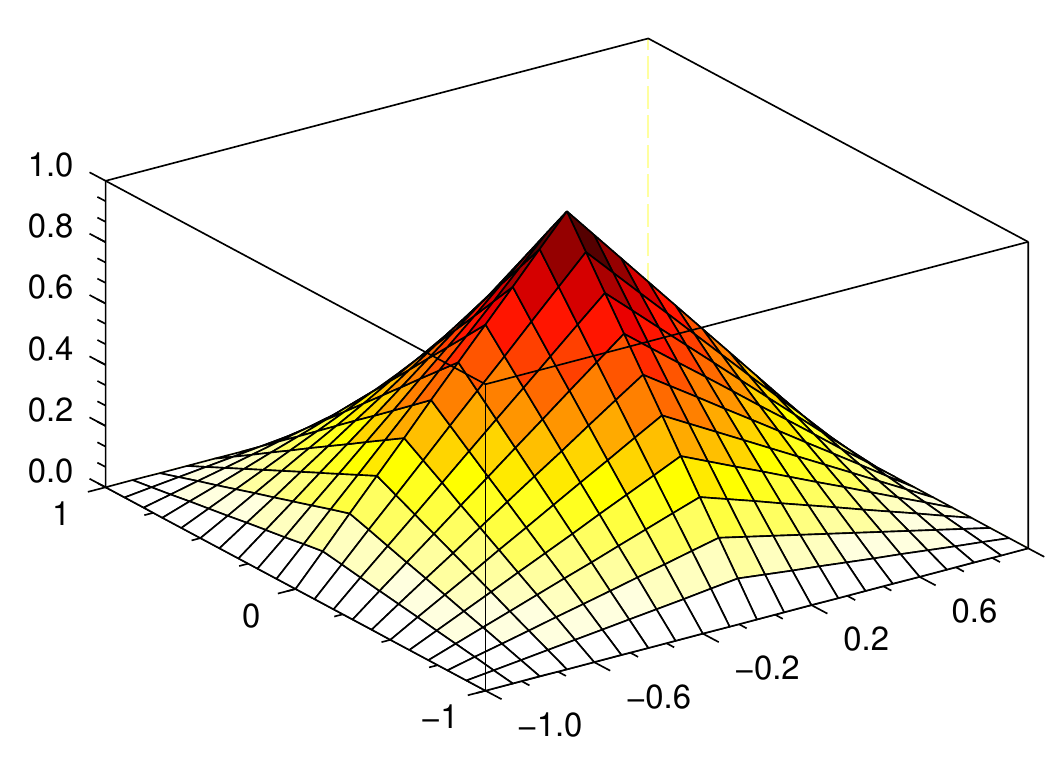}}
\put(22,10){$\eta$}
\put(87,5){$\xi$}
\put(0,44){$v$}
\end{picture}
\caption{basis `pyramid'  for the null space of the adjoint operator on an element: $v=(1-|\xi|)(1-|\eta|)=(1-|x-X_i|/H)(1-|y-Y_i|/H)$.}
\label{fig:zadjoint}
\end{figure}

The solvability condition is then that the integral of the subgrid residuals of the \pde\ with this~$v$ over the $i,j$th~element determines a correction to the model evolution.  

It will not escape your notice that the solvability condition integral parallels integrals in the Galerkin finite element method.
Thus view the Galerkin finite element method as a leading approximation to our systematic slow manifold closure of discrete modelling.

Analogous arguments derive the adjoint for patch dynamics.
The version of boundary conditions~\eqref{eq:ccdbc} for patches refers to the grid value of the field on the right-hand side.
With care integrating over a patch, the linear constant diffusion operator on the right-hand side of \pde~\eqref{eq:ccdpde} has adjoint
\begin{equation}
\delsq v+\left(\int_{\partial E_{i,j}}\D nv\,ds\right)\delta_{\text{Dirac}}(x-X_i,y-Y_j)\,,
\quad \text{such that }
v=0\text{ on }\partial E_{i,j}\,.
\label{eq:patadj}
\end{equation}
The point source at the central grid point makes up for the diffusive loss of material across the edge of the patch to lead to a null vector of the adjoint with a logarithmic singularity at the grid point.
We have not yet used this interesting adjoint for patch dynamics.

\paragraph{Return to the discrete Ginzburg--Landau model}
Computer algebra readily computes the subgrid residuals of the \glpde{} to the next higher order (the code is detailed elsewhere~\cite[\S2.3]{MacKenzie11a}).
Taking the inner product with the adjoint null vector~$v$, and remembering contributions from the inter-element coupling conditions~\eqref{EbcsdL}, gives the next higher order discrete model.
%
Because of the faithful resolution of subgrid structures and inter-element interactions, the resulting discrete models are algebraically extremely complicated.
Thus users may prefer, as they often do now, to use models of the nonlinear dynamics of lower order.
Then higher order discretisations derived via this approach provide local estimates of the local error in a lower order simulation as it is computed on the fly.

\section{Generally compute multi-dimensional subgrid fields numerically}
\label{chapnumcm}

Here the subgrid field of the slow manifold is constructed numerically for the example \glpde{} in~2D.
The approach generalises straightforwardly to a wide range of nonlinear \pde{}s in multiple dimensions, such as the general reaction-diffusion \pde~\eqref{Erde}.

New complexities arise.
Although the spatial structure is obtained numerically, the subgrid field must be also parametrised by the grid values~$\uij$, the inter-element\slash patch coupling parameter~$\gamma$, and the nonlinear parameter~$\alpha$.
Therefore,  the construction involves symbolic parameters.
The algorithm required to develop the holistic model must efficiently solve the corresponding mixed numerical and symbolic problem.
Only then will the dynamical systems methodology be able to be usefully employed in modelling general \pde{}s such as~\eqref{Erde}.
\emph{The focus of this section is on this novel combined algebraic\slash numerical construction of the subgrid fields.} 

Numerical construction of the subgrid field introduces errors which are separate from the orders of errors in approximating the slow manifold.
These errors from the numerical construction of the holistic discretisation are a major concern.
The numerical construction of the subgrid field and its evolution has challenging details: \S\ref{S_num_2D_ext},~How should the subgrid problem be solved?  \S\ref{S_num_2D_perf},~What subgrid resolutions will accurately reproduce the analytical holistic models?  \S\ref{S_num_eff},~What is an efficient implementation?

This section does not compare the effectiveness of the holistic modelling with traditional approaches, that effectiveness has previously been reasonably established in 1D~\cite[e.g.]{MacKenzie05a, Roberts98a, Roberts04d}.  Instead this section focusses on developing the methodology needed to use the approach for more challenging \pde{}s in multiple dimensions and in the challenge of supporting the potential efficiencies of the gap-tooth scheme of the next section~\ref{sec:patch}.

\subsection{Outline the numerical slow manifold in 2D}
\label{S_num_2D_ext}
In each element\slash patch we discretise the microscale subgrid as shown in Figure~\ref{fig:2D_sub}.
Whereas previously we sought the field~$u_{i,j}(x,y,t)$ in the $i,j$th~element, here we seek the microscale grid values in the $i,j$th~element\slash patch.  Consequently, here we index the subgrid by variables $k$~and~$\ell$ where $|k|,|\ell|<n$\,. For example Figure~\ref{fig:2D_sub} shows the case $n=4$\,; the subgrid is shown solid (blue\slash magenta) for this particular example of a $9\times9$ subgrid.
The subgrid field extends to $X_{i}\pm rH$~and~$Y_{j}\pm rH$ to apply either of the 2D non-local \icc~\eqref{EbcsdL} or~\eqref{eq:epcc}.
At the $k,\ell$th~subgrid point of the $i,j$th~element\slash patch we define the microscale grid value~$u_{i,j,k,\ell}(t)$ for subgrid indices $|k|,|\ell|<n$\,.  In this section we invoke centre manifold theory to analyse the microscale evolution of~$u_{i,j,k,\ell}(t)$ in order to model the macroscale dynamics.
 
\begin{figure}
\begin{center}
\footnotesize \setlength{\unitlength}{0.23em}
\begin{picture}(80,80)
\multiput(10,0)(20,0){4}{
  \multiput(0,0.5)(0,3){27}{\line(0,1){2}}}
\multiput(0,10)(0,20){4}{
  \multiput(0.5,0)(3,0){27}{\line(1,0){2}}}
\multiput(20,20)(20,0){3}{
\multiput(0,0)(0,20){3}{\circle*{2}} } 
\put(12,15){$i-1,j-1$}
\put(34,15){$i,j-1$} \put(52,15){$i+1,j-1$} \put(15,35){$i-1,j$}
\put(38,35){$i,j$} \put(55,35){$i+1,j$} \put(12,55){$i-1,j+1$}
\put(34,55){$i,j+1$} \put(52,55){$i+1,j+1$}
\put(30,6){\vector(1,0){20}}
\put(50,6){\vector(-1,0){20}}
\put(40,3){$H$}
\color{blue}
\multiput(20,20)(5,0){9}{\line(0,1){40}}
\multiput(20,20)(0,5){9}{\line(1,0){40}}
\end{picture}
\quad
\begin{picture}(80,80)
\multiput(10,0)(20,0){4}{
  \multiput(0,0.5)(0,3){27}{\line(0,1){2}}}
\multiput(0,10)(0,20){4}{
  \multiput(0.5,0)(3,0){27}{\line(1,0){2}}}
\multiput(20,20)(20,0){3}{
\multiput(0,0)(0,20){3}{\circle*{2}} } 
\put(0,-2){
\put(12,15){$i-1,j-1$}
\put(34,15){$i,j-1$} \put(52,15){$i+1,j-1$} \put(15,35){$i-1,j$}
\put(38,35){$i,j$} \put(55,35){$i+1,j$} \put(12,55){$i-1,j+1$}
\put(34,55){$i,j+1$} \put(52,55){$i+1,j+1$}
}
\put(30,6){\vector(1,0){20}}
\put(50,6){\vector(-1,0){20}}
\put(40,3){$H$}
\color{magenta}
\newcommand{\patch}{
\multiput(0,0)(1,0){9}{\line(0,1){8}}
\multiput(0,0)(0,1){9}{\line(1,0){8}}
}
\multiput(16,16)(20,0){3}{
\multiput(0,0)(0,20){3}{\patch} } 
\end{picture}
\end{center}
\caption{Example of the $n=4$ subgrid on 2D elements (left) and patches (right): a subgrid labelled as ``$n=4$'' extends to be $9\times9$ as a microscale subgrid extends to~$\pm n$ on all sides of a macroscale grid point.}
\label{fig:2D_sub}
\end{figure}
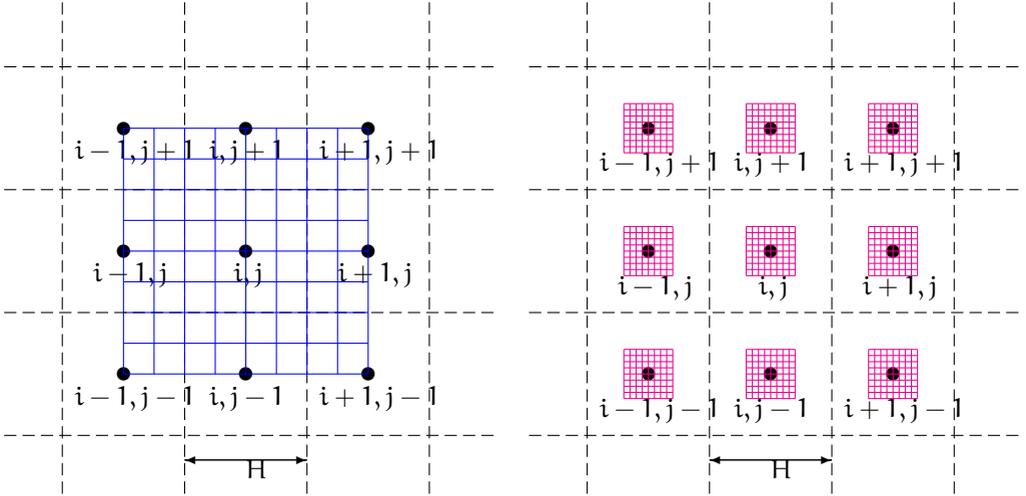

After discretising the subgrid microscale field in the $i,j$th~element\slash patch, classic finite differences approximate the spatial derivatives of the subgrid field of the \glpde:
\begin{align}
\dot u_{i,j,k,\ell}
=&\frac{u_{i,j,k+1,\ell}+u_{i,j,k-1,\ell}+u_{i,j,k,\ell+1}+u_{i,j,k,\ell-1}
-4u_{i,j,k,\ell}}{(rH/n)^2} 
\nonumber\\&{}
+\alpha(u_{i,j,k,\ell}-u_{i,j,k,\ell}^3).
\label{eq:nde}
\end{align}
These microscale discretised equations are solved at each of the subgrid points inside each of the elements.
The overlapping elements are coupled with \icc\ corresponding to~\eqref{EbcsdL}, namely
\begin{equation}
\begin{cases}
u_{i,j,\pm n,\ell}=\gamma u_{i\pm1,j,0,\ell}+(1-\gamma)u_{i,j,0,\ell}\,,
& |\ell|<n\,,
\\
u_{i,j,k,\pm n}=\gamma u_{i,j\pm1,k,0}+(1-\gamma)u_{i,j,k,0}\,,
& |k|<n\,,
\end{cases}
\label{eq:nibcy}
\end{equation}
or in the case of disjoint patches by \icc\ corresponding to~\eqref{eq:epcc}, namely
\begin{equation}
\begin{cases}
u_{i,j,\pm n,\ell}=\shift _i^{\pm r}(\gamma)\shift _j^{r\ell/n}(\gamma) u_{i,j,0,0}\,,
& |\ell|<n\,,
\\
u_{i,j,k,\pm n}=\shift _i^{rk/n}(\gamma)\shift _j^{\pm r}(\gamma) u_{i,j,0,0}\,,
& |k|<n\,.
\end{cases}
\label{eq:nibcp}
\end{equation}
Equations~\eqref{eq:nde}--\eqref{eq:nibcy}/\eqref{eq:nibcp} form a system on a microscale lattice.
The same centre manifold theorems apply to the system~\eqref{eq:nde}--\eqref{eq:nibcy}/\eqref{eq:nibcp} to assure us of the existence, relevance and construction of a slow manifold, macroscale discrete model of the lattice dynamics.
Indeed the application of centre manifold theory is more straightforward as the microscale lattice system is finite dimensional, in contrast to the `infinite' dimensionality of a microscale \pde.
The macroscale slow manifold to construct is that the subgrid field $u_{i,j,k,\ell}=u_{i,j,k,\ell}(\vec U,\gamma,\alpha)$ where, defining the grid value $\uij=u_{i,j,0,0}$\,, the macroscale grid values~$\vec U$ evolve according to the system $\dotuij=g_{i,j}(\vec U,\gamma,\alpha)$.

An iteration scheme finds the microscale subgrid field and the macroscale slow evolution.
We leave the case of patches until Section~\ref{sec:patch} and here focus upon the case of overlapping elements.
The initial approximation is that of a constant field in each element: $u_{i,j,k,\ell}\approx \uij$ such that $\dotuij=g_{i,j}\approx 0$\,.
Given any current approximation, $\vec u_{i,j}$~and~$g_{i,j}$, we seek an improved approximation to the field, $\vec{u}_{i,j}+\vec  u_{i,j}'$, and evolution, ${g}_{i,j}+g_{i,j}'$, where $\vec u_{i,j}'$~and~$g_{i,j}'$ are corrections to be found in each iteration.
At each iteration, and in the case of overlapping elements, the following linear equations driven by the current residuals are solved for the corrections~$\vec{u}_{i,j}'$ and~$g_{i,j}'$
\begin{align}
&\left(\frac n{rH}\right)^2\left(u'_{i,j,k+1,\ell}+u'_{i,j,k-1,\ell}
+u'_{i,j,k,\ell+1}+u'_{i,j,k,\ell-1}
-4u'_{i,j,k,\ell}\right) -g_{i,j}'=\Res_{\ref{eq:nde}}\,,
\nonumber\\&
u'_{i,j,\pm n,\ell}-u'_{i,j,0,\ell}=\Res_{\ref{eq:nibcy}}\,,\quad
u'_{i,j,k,\pm n}-u'_{i,j,k,0}=\Res_{\ref{eq:nibcy}}\,,\quad
u'_{i,j,0,0}=0\,.
\label{E_glnum_eqs}
\end{align}
The iteration repeats until all residuals are zero to a specified order of error.
This iteration scheme follows that for the analytic construction of the subgrid field and is documented in full detail in a separate technical report~\cite[\S4]{MacKenzie11a}.
Centre manifold theory then assures us that the resultant macroscale discrete model $\dotuij =g_{i,j}(\vec U,\gamma,\alpha)$ is accurate to the same order of error.

For example, for the coarsest possible subgrid field, the case $n=2$ when the microscale subgrid within each element has half the spacing of the macroscale grid, the lattice dynamics on the microscale subgrid gives the low order, macroscale model
\begin{align}
\dotuij =&\frac{\gamma}{H^2}\bdelta^2 \uij +\alpha\left(\uij - \uij^3\right)
\nonumber\\&{}
-\frac{\gamma^2}{16H^2}\bdelta^4 \uij
+\alpha \gamma \left(\rat1{16} \bdelta^2\uij^3-\rat3{16} \uij^2 \bdelta^2\uij\right)
+\Ord{\gamma^3+\alpha^3}\,,
\label{E_num_gl_g1r1_2g}
\end{align}
Compare this model with the analytic macroscale discrete model~\eqref{E_gl2d_g1a1}: the terms in the first line are identical; the same higher order terms appear in the second line but the coefficients differ by~$25$\%.
This correspondence is promising for such a coarse microscale subgrid discretisation.

Recall that this approach models the dynamics on a microscale lattice.
Thus this approach transforms microscale lattice dynamics onto a macroscale lattice.
Such transformation across scales of lattices may be repeated across an entire hierarchy of lattices to form \emph{dynamics} on a multigrid~\cite{Roberts08c}.
The consistency theorems of section~\ref{sec:nccec} provide additional suuport for the modelling across a hierarchy of lattices.

\subsection{Low resolution subgrids are accurate in 2D}
\label{S_num_2D_perf}

How do macroscale models constructed via a numerical microscale, such as~\eqref{E_num_gl_g1r1_2g}, compare with analytic macroscale models? We compare in two ways: one via the convergence of the coefficients; and the other by the accuracy of the predicted bifurcation diagrams.
It appears that the microscale subgrid need not be of high resolution.

\begin{table}
\caption{coefficients in the~$\Ord{\gamma^3+\alpha^3}$ models, such as~\eqref{E_num_gl_g1r1_2g}, evidently converge to the correct analytic coefficients, in~\eqref{E_gl2d_g1a1} and labelled~$\infty$ in the table, with errors~$\Ord{1/n^2}$ as the resolution of the microscale grid improves.}
\label{tbl:num_gl_g1r1_2g}
\begin{center}
\begin{tabular}{ccccc}
$n$ & subgrid &
$\gamma^2\bdelta^4\uij/H^2$ & 
$\alpha\gamma\bdelta^2\uij^3$ & 
$\alpha\gamma\uij^2\bdelta^2\uij$ \\ \hline
\vphantom{$\Big|$}
$2$ & $5\times5$ & $-\frac1{16}$ & $\frac1{16}$ & -$\frac3{16}$ \\
\vphantom{$\Big|$}
$4$ & $9\times9$ & $-\frac5{64}$ & $\frac5{64}$ & -$\frac{15}{64}$ \\
\vphantom{$\Big|$}
$8$ & $17\times17$ & $-\frac{21}{256}$ & $\frac{21}{256}$ & -$\frac{63}{256}$ \\
\hline
\vphantom{$\Big|$}
$\infty$ && $-\frac1{12}$ & $\frac1{12}$ & -$\frac1{4}$ \\
\hline
\end{tabular}
\end{center}
\end{table}

First look at the coefficients of the~$\Ord{\gamma^3+\alpha^3}$ models such as~\eqref{E_num_gl_g1r1_2g}.
Recall that the number of microscale subgrid points, from one macroscale grid point to the next, in each dimension, is~$n$.
The coefficients linear in the coupling parameter~$\gamma$ and nonlinearity~$\alpha$ are exact for all $n\geq 2$\,, only the higher order coefficients vary with subgrid resolution.
Table~\ref{tbl:num_gl_g1r1_2g} tabulates coefficients in these nonlinear terms, those of~$\Ord{\gamma^2+\alpha^2}$ in models such as~\eqref{E_num_gl_g1r1_2g}, for some values of~$n$.
Evidently the coefficients converge to the continuum \pde\ values with error~$\Ord{1/n^2}$.
We expect such quadratic convergence from the quadratic  approximation in~\eqref{eq:nde} of the subgrid scale \pde\ dynamics.

\begin{table}
\caption{maximum errors in the coefficients of the $\Ord{\gamma^4+\alpha^4}$ model when approximated numerically at three different subgrid resolutions.
The decrease by at least a factor of four, upon doubling~$n$, indicates quadratic convergence.}
\label{tbl:me2dgl3}
\begin{center}
\begin{tabular}{llllll}
$n$ & $\gamma^2/H^2$ & $\gamma\alpha$ & $\gamma^3/H^2$ & $\gamma^2\alpha$ & $\gamma\alpha^2H^2$ \\ \hline
$2$ & $0.021 $ & $0.062 $ & $0.0033 $ & $0.14 $ & $0.0016   $ \\
$4$ & $0.0052$ & $0.016 $ & $0.00086$ & $0.040$ & $0.000098 $ \\
$8$ & $0.0013$ & $0.0039$ & $0.00022$ & $0.010$ & $0.0000061$ \\ \hline
\end{tabular}
\end{center}
\end{table}

Similarly, we compare numerically obtained coefficients for the $\Ord{\gamma^4+\alpha^4}$ model.
Because of the complexity of the model we make a limited comparison: for each order in $\alpha$~and~$\gamma$, Table~\ref{tbl:me2dgl3} reports the largest error in the numerically obtained coefficient for three different subgrid scale resolutions.
Evidently, these maximum errors decrease like~$1/n^2$ to confirm the accuracy of the numerical description of the subgrid scale dynamics. 


\begin{figure}
\centering
\includegraphics[width=\textwidth]{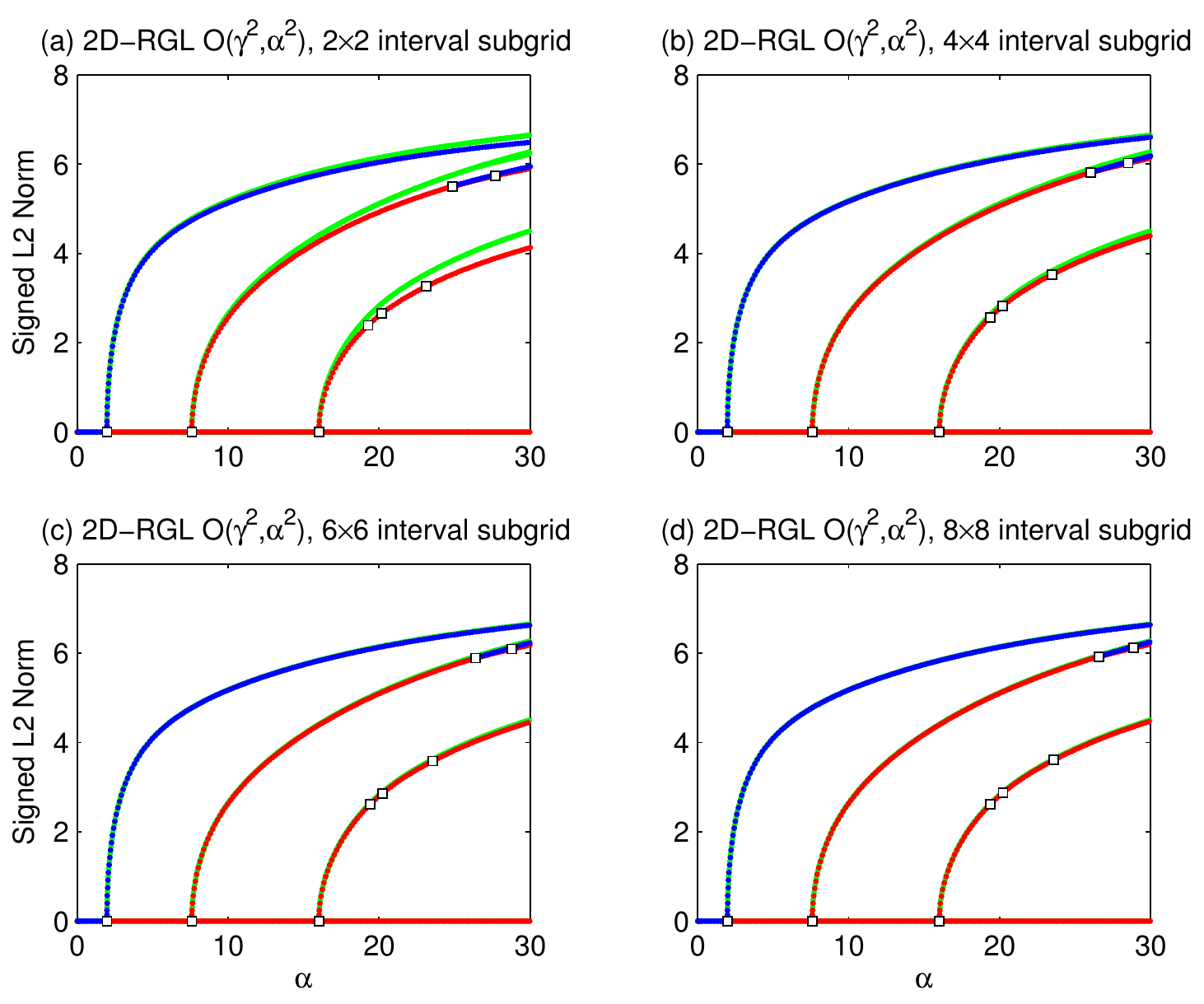}
\caption{Bifurcation diagrams of the $O(\gamma^2,\alpha^2)$ holistic
models of the 2D Ginzburg--Landau system with $8\times 8$ macroscale
elements on $[0,\pi]\times[0,\pi]$ for subgrids of
(a)~$5\times5$, (b)~$9\times9$, (c)~$13\times13$ and (d)~$17\times17$.
The bifurcation diagram for the analytically constructed
model is shown in green.}
\label{ginz2d_num_g1r1}
\end{figure}

Second, we turn to the bifurcation diagram to see the sort of errors incurred in using the approximate models.
Figure~\ref{ginz2d_num_g1r1} shows the bifurcation diagrams for the $\Ord{\gamma^2,\alpha^2}$ holistic model of the Ginzburg--Landau system for four subgrid resolutions.
Here the equilibria shown in green are not the accurate solution of the Ginzburg--Landau system, but rather the equilibria of the analytic $\Ord{\gamma^2,\alpha^2}$ holistic model in 2D (obtained from~\eqref{E_gl2d_g1a1} by omitting the $\gamma^2$~term).
Observe that with a subgrid resolution of just $4\times4$ intervals the bifurcation diagram for the numerically constructed $\Ord{\gamma^2,\alpha^2}$ holistic model is almost indiscernible from the analytic model over nonlinearity $0\le\alpha\le20\,$.
Higher subgrid resolutions are indiscernible to even larger nonlinearity~$\alpha$.

As a last comparison of bifurcation diagrams, Figure~\ref{ginz2d_num_8_bif_2} shows one example confirming  that by computing to higher order in the inter-element interactions, and resolving the subgrid scale structures numerically, the predictions of the numerically derived models do improve significantly. 

Numerically resolving the microscale subgrid structures does generate usefully accurate, slow manifold, macroscale, discrete models.

\subsection{An efficient computer algebra approach is crucial}
\label{S_num_eff}

The difficulty associated with the numerical construction of the subgrid field is the mixed numerical and symbolic nature of the equations involved in the iteration scheme.
The size of the system of equations increases as the subgrid resolution improves, and the complexity of the symbolic nonlinear residuals increases quickly as higher order holistic models are constructed.
However, these comments apply to, and this subsection only addresses, the pre-simulation preprocessing step of constructing the slow manifold discretisations: the computation time of the derived macroscale closure on the macroscale domain is as normal for the method of lines, and is independent of the considerations here.

Computer algebra packages such as Reduce~\cite{Reduce04} or Mathematica~\cite{Wolfram96} have general routines that solve systems of equations such as~\eqref{E_glnum_eqs}.
However, these \verb|solve| routines are inefficient for the many equations and complicated expressions here.
Even a low resolution, $n=2$ subgrid, took many minutes using the built in \verb|solve| routines of both Reduce and Mathematica.
Instead we develop an approach that is practical for implementation with a large number of complicated symbolic terms.  Nonetheless, different computer algebra packages execute this preprocessing step in a time that differs by an order of magnitude: users need to know that so far we find Reduce~\cite{Reduce04} to be the most efficient.

\paragraph{Transform to constant coefficient}
Recall that at each step of the iteration scheme we solve a problem for updates to the subgrid field $\vec{u}'_{i,j}$ and its evolution~$g_{i,j}'$.
Multiply the first (field) equation in~\eqref{E_glnum_eqs} by~$r^2H^2$ and replace~$g'_{i,j}$ by $\mathcal G'=r^2H^2g'_{i,j}$\,.
Then the left-hand side of the new form of the equations has numerical constant coefficients; algebraic expressions only occur in the right-hand side.

Further, the left-hand side of the new equations remain the same for every iteration.
Consequently, the first iteration constructs an \textsc{lu}~factorisation of the left-hand side, which is then used to solve equation~\eqref{E_glnum_eqs} for updates in every iteration~\cite[\S4]{MacKenzie11a}.
The \textsc{lu}~factorisation is performed once and requires approximately $\frac13 N^3$~operations~\cite[e.g.]{Press92}.
Here the number of equations for the subgrid structure are $N=(2n+1)^2+1$\,; for example, $N=290$ for the $n=8$ microscale subgrid.
At each step of the iteration scheme the \textsc{lu}~factorisation algorithm operates on the symbolic residual vector.
We suspect 2D and 3D problems would be solved more efficiently through sparse methods such as iterative multigrid~\cite{mccormick92, brigg00} or incomplete \textsc{lu}~factorisation and Krylov subspace methods~\cite{kelley95, vorst03}.
Such alternatives remain for later exploration.

\begin{table}
\caption{Reduce and Mathematica computational times for numerical construction of $\Ord{\gamma^4,\alpha^2}$ holistic models of the one dimensional Ginzburg--Landau equation for various subgrid scale resolutions,~$n$.}
\label{table_comptime_red}
\begin{center}
\begin{tabular}{rrr}
$n$&Reduce&Mathematica\\
\hline
$2$&$1.1$\,s&$70.2$\,s\\
$4 $&$3.1$\,s&$215.4$\,s\\
$8 $&$8.3$\,s&$367.6$\,s\\
$16 $&$23.7$\,s&$517.7$\,s\\
\hline
\end{tabular}
\end{center}
\end{table}

\paragraph{Reduce was much faster}
Computational experiments found that the computer algebra package Reduce was at least an order of magnitude faster than Mathematica.
Table~\ref{table_comptime_red} lists the computational time for the Reduce and the Mathematica implementation for constructing $\Ord{\gamma^4,\alpha^2}$~holistic models of the \emph{one dimensional} Ginzburg--Landau equation with subgrid resolutions of~$2$, $4$, $8$ and~$16$ subgrid intervals.
These times were observed on a Pentium~III, 750\,MHz processor, with 256\,Mb~\textsc{ram}, running Reduce~3.7, under Windows~XP.
Table~\ref{table_comptime_red} shows the Reduce implementation was $20$--$70$ times faster than the Mathematica implementation  (even with the repeated help of the Mathematica news group).
Thus we use the free package Reduce~\cite{Reduce04}.


\section{The slow manifold of emergent patch dynamics}
\label{sec:patch}

Now return to the gap-tooth dynamics on patches such as the simulations of the \glpde\ shown in Figures~\ref{fig:Matlab/early5-4-8-12-10D/0}--\ref{fig:Matlab/early5-4-8-12-10D/1}.  The challenge is to deduce by analysis some of the important properties of such `equation free' simulation so we understand its performance in applications.  

This section shows how the emergent patch dynamics may be viewed as a slow manifold closure that is also consistent with the underlying microscale model.  Hence, for example, patch dynamics could explore the interesting domain competition inherent in the Ginzburg--Landau \pde.

\paragraph{Subgrid microscale modes form very rapid transients}

Figures~\ref{fig:Matlab/early5-4-8-12-10D/0}--\ref{fig:Matlab/early5-4-8-12-10D/1} sample the evolution from a non-smooth initial condition of the \glpde.  The transients are the rapid diffusive smoothing of the initial jagged microscale structures within each spatial patch.  Thereafter we would see the slow macroscale evolution of domain competition, underpinned by smooth microscale structures within each patch.  Thus in the long term we see relatively slow evolution of smooth structures in each patch.

\begin{figure}
\centering
\includegraphics{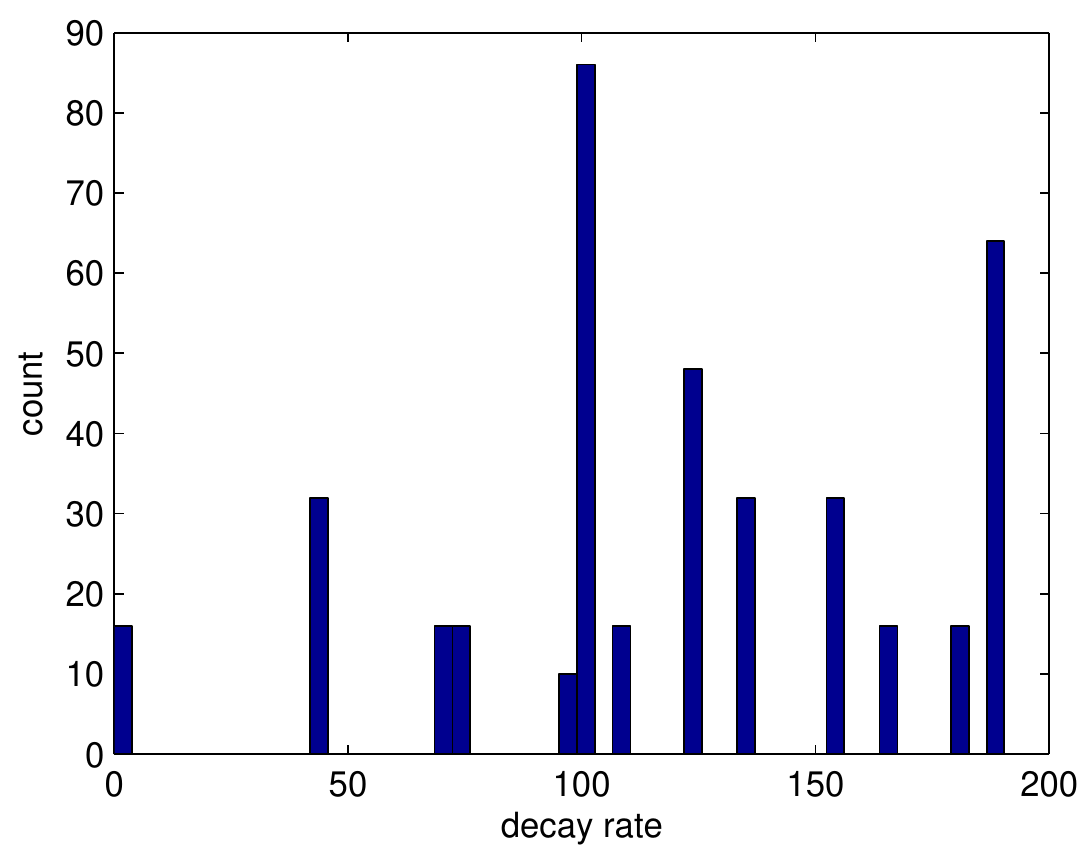}
\caption{histogram of the number of eigenvalues of linear diffusion--- $\alpha=0$ in the \glpde---showing clear time scale separation between the $16$~global modes with small eigenvalues, and the multitude of microscale modes decaying at rates${}>40$.  Here the gap ratio $r=1/4$ and macroscale spacing $H=2$.}
\label{fig:patcheigenv}
\end{figure}

Eigenanalysis of linear diffusion---$\alpha=0$ in the \glpde---confirms this separation of modes into rapidly decaying subgrid modes and long lasting macroscale modes.  For example, consider doubly periodic diffusion on $4\times 4$ patches on a macroscale grid with spacing $H=2$ with each patch composed of a $5\times 5$ microscale grid.  Couple patches with the \icc~\eqref{eq:epcc}.  Then numerical differentiation of the code that simulates Figures~\ref{fig:Matlab/early5-4-8-12-10D/0}--\ref{fig:Matlab/early5-4-8-12-10D/1} generates the matrix of the diffusion dynamics on these patches.  Figure~\ref{fig:patcheigenv} plots a histogram of the eigenvalues of this matrix.  Most eigenvalues are large and negative corresponding to the rapid diffusive decay of subpatch structures.  However, $16$~eigenvalues are near zero, one for each of the $16$~patches.  These near zero eigenvalues characterise long lasting, emergent, global modes.

\begin{figure}
\centering
\includegraphics{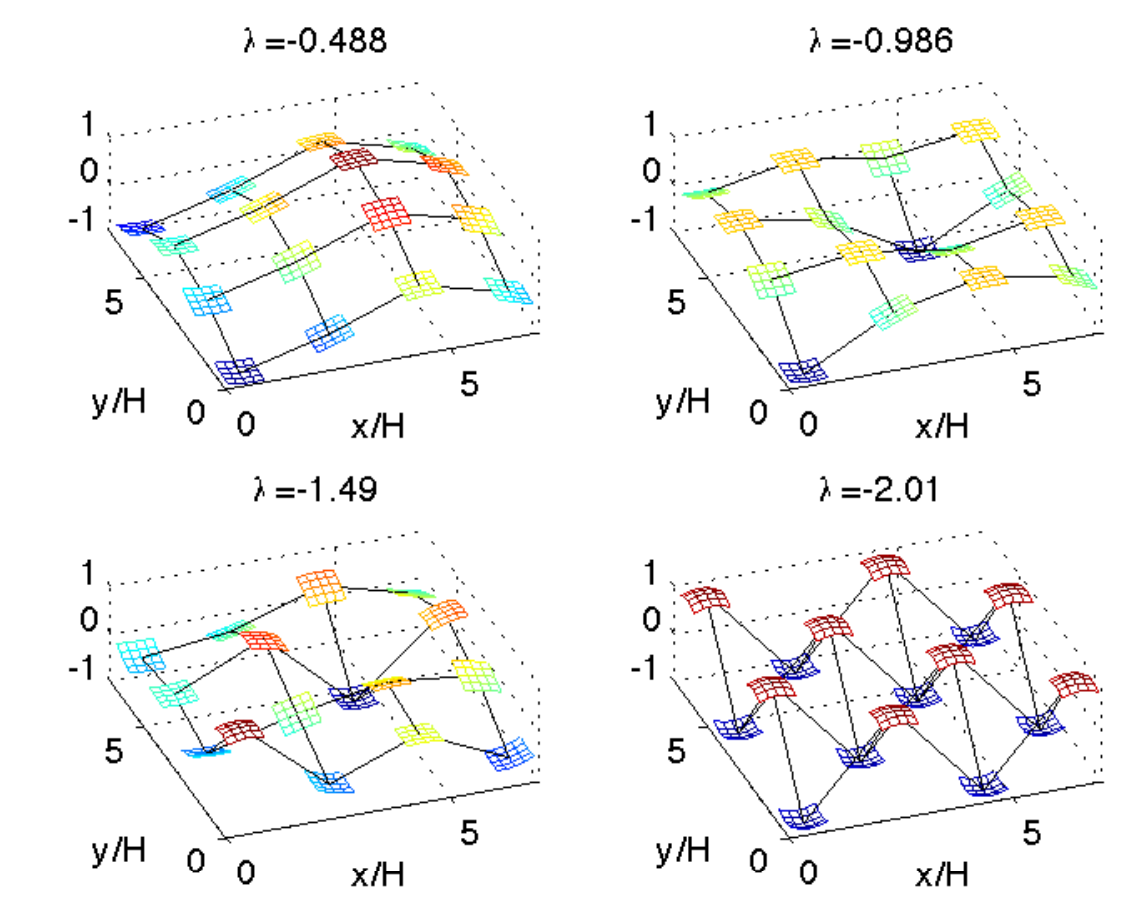}
\caption{some global eigenmodes labelled by their eigenvalue.  Here the ratio $r=1/4$ and macroscale spacing $H=2$.  For comparison, a second order, finite difference scheme with corresponding grid spacing $H=2$ would have eigenvalues $\lambda=-1/2,-1,-3/2,-2$: the shown macroscale eigenvalues correspond well with these finite difference ones.}
\label{fig:patcheigenm}
\end{figure}

Figure~\ref{fig:patcheigenm} confirms the smooth global modes.  Not discernible in the eigenvalues of Figure~\ref{fig:patcheigenv} is that the small eigenvalues form four main groups.   Figure~\ref{fig:patcheigenm} plots a representative eigenmode from each of these four groups: from the near~$\sin x$ mode of $\lambda=-0.471$ to the macroscopic zig-zag but microscopic smooth mode of $\lambda=-2.05$\,.   This linear analysis confirms the basis for theoretical support of nonlinear patch dynamics for general reaction-diffusion \pde{}s~\eqref{Erde}.

\paragraph{Centre manifold theory supports macroscale closure}
Section~\ref{S_2D_divide} introduces~$\gamma$ to parametrise the coupling between patches: $\gamma=1$ corresponds to the coupling used in `equation free' simulation, but when $\gamma=0$ the eigenvalues are perturbed so that the small eigenvalues, such as those in Figure~\ref{fig:patcheigenv}, become precisely zero.  Then Corollaries~\ref{cor:exist} and~\ref{cor:rc} establish the existence of emergent patch dynamics as a slow manifold at finite coupling~$\gamma$ and finite nonlinearity~$\alpha$ for general reaction-diffusion \pde{}s~\eqref{Erde}.

\paragraph{Patch dynamics are consistent with the microscale}
Section~\ref{chapnumcm}  describes how to construct slow manifold discretisations of general reaction-diffusion \pde{}s~\eqref{Erde}.  For example, consider the \glpde\ when discretised inside patches by a $(2n+1)\times(2n+1)$ microscale lattice. Using the patch \icc~\eqref{eq:epcc} computer algebra computes the slow manifold discretisation is
\begin{align}
\dotuij={}&
\frac\gamma{H^2}\bdelta^2\uij
-\frac{\gamma^2}{12H^2}\big(1-\rat{r^2}{n^2}\big)\bdelta^4\uij
+\frac{\gamma^3}{90H^2}\left(1-\rat{r^2}{n^2}\right)\left(1-\rat{r^2}{4n^2}\right)\bdelta^6\uij
\nonumber\\&{}
+\alpha(\uij-\uij^3)
+\alpha\gamma r^2c_n\left[ \bdelta^2(\uij^3)-3\uij^2\bdelta^2\uij\right]
\nonumber\\&{}
+\Ord{\alpha^2+\gamma^4},
\label{eq:glsm9}
\end{align}
for operator~$\bdelta$ defined by~\eqref{eq:bdelta}.  The nonstandard discretisation of the cubic reaction, appearing via the $\alpha\gamma$~term above, is obtained from a systematic approximation of the subgrid, microscale, out of equilibrium, structures: its coefficient varies with the $(2n+1)\times(2n+1)$ microscale lattice as
\begin{equation*}
c_n\to 0.0758188\text{ as }n\to\infty\,, 
\qtq{and}
c_n\approx \frac{751-739/n^2}{6(1651+191/n^2)}
\end{equation*}
exactly for $n=2,3,4$ and to four significant digits for $n=5:8$\,.  Evaluated at full coupling $\gamma=1$ corresponding to the `equation free' gap-tooth simulation, the discretisation~\eqref{eq:glsm9} is consistent with the microscale lattice discretisation of the \glpde.  Replacing the differences in~\eqref{eq:glsm9} by their expansion in derivatives, the equivalent \pde\ to~\eqref{eq:glsm9} is
\begin{align}
\D tU={}&
\nabla^2U
+\frac{H^2r^2}{12n^2}\left(\Dn x4U+\Dn y4U\right)
+\frac{H^4r^4}{360n^4}\left(\Dn x6U+\Dn y6U\right)
\nonumber\\&{}
+\alpha(U-U^3)
+{\alpha H^2r^26c_n}{}U\left[\Big(\D xU\Big)^2+\Big(\D yU\Big)^2\right]
\nonumber\\&{}
+\Ord{\alpha^2+H^6}.
\label{eq:glsm9pde}
\end{align}
This \pde\ is not the \glpde\ because the underlying dynamics are those of the microscale discretisation: instead this \pde\ is equivalent to that of the microscale discretisation on the fine lattice which has spacing~$rH/4$.   Such consistency of macroscale patch dynamics with the fine scale dynamics is established for a range of \pde\ operators in Theorem~\ref{thm:ipc} of the next Section~\ref{sec:nccec}.

Lastly, these patch dynamics connect to classic finite difference discretisations.   Fix the macroscale grid spacing~$H$, but let the patch size and the microscale grid spacing become small via the ratio $r\to0$\,.  Then the slow manifold model~\eqref{eq:glsm9} reduces to the classic finite difference approximation of the \glpde: the equivalent \pde~\eqref{eq:glsm9pde} is the \glpde\ as patch size $r\to0$.  Using classic interpolation to provide coupling conditions~\eqref{eq:epcc} for tiny patches is equivalent to classic finite differences of the microscale dynamics.  Thus when we apply the `equation free' gap-tooth method on simulators for which we do not know the microscale equations, we will nonetheless obtain a macroscale simulation which is consistent with the unknown microscale equations.

Furthermore, this consistency of the macroscale simulation holds no matter how small the patches, here parametrised by the ratio~$r$.  The constraints on the macroscale grid will be those familiar to anyone using classic finite difference or finite element approximations: namely, the macrogrid step~$H$ must be small enough to resolve the macroscale variations.  But the patch size can be vanishingly small making the gap-tooth method extremely efficient for those systems with extremely wide separation of scales.

\section{Non-local coupling conditions enforce consistency}
\label{sec:nccec}

Recall that the constructed holistic models of the \glpde\ are consistent with the \pde\  as the grid size $H\to0$\,, see equations \eqref{E_gl2d_epde} in Section~\ref{S_2D_low}.
Now we prove that general consistency follows from the specific choice of nonlocal inter-element\slash patch coupling conditions~\eqref{EbcsdL} and~\eqref{eq:epcc}.

We start with a similar theorem to one previously proved for the consistency of holistic discretisation in one space dimension~\cite{Roberts00a}.
The critical innovation here is in the proof: previous proofs were constructive whereas here it is not.
Avoiding a constructive proof is essential here as we do not know analytic forms for the slow manifold subgrid fields in multiple dimensions.
In this new proof: a new lemma establishes consistency for nonlinear reaction-diffusion \pde{}s in~$1$D; and an immediate new corollary then proves consistency of the 2D~holistic discretisation of a wide range of 2D~nonlinear, reaction-diffusion \pde{}s.

\begin{theorem}[1D linear consistency] \label{thm:1dc}
Consider the \pde\ $\partial_tu=\cL u$ for some local, isotropic, homogeneous, linear operator~$\cL$.
Model the dynamics on overlapping elements of an equi-spaced grid $X_i=iH$\,.
Let $u_i(x,t)$~denote the subgrid field in the $i$th~element satisfying the \pde\ $\partial_tu_i=\cL u_i$ on the interval~$(X_{i-1},X_{i+1})$ with the moderated inter-element coupling conditions
\begin{equation}
u_i(X_{i\pm1},t)=\gamma u_{i\pm1}(X_{i\pm1},t)+(1-\gamma)u_i(X_i,t)\,.
\label{eq:1dibc}
\end{equation}
When inter-element interactions are truncated to residuals~$\Ord{\gamma^p}$ the grid values $U_i(t)=u_i(X_i,t)$, at full coupling $\gamma=1$\,, evolve consistently with the \pde\ $\partial_tu=\cL u$ to an order in grid size~$H$ that increases with~$p$. 
\end{theorem}

\begin{proof}
We proceed with some classic operator algebra~\cite[e.g.]{npl61}.
The principle obstacle is to transform subgrid spatial differences, indicated by subscript~$x$, into macroscale grid differences, indicated without a subscript.
Begin with the \pde\ on the $i$th~element: $\partial_t u_i=\cL u_i$\,.
Because the operator~$\cL$ is isotropic and homogeneous it may be formally expanded in even centred differences as
\begin{displaymath}
\cL  =\sum_{k=0}^\infty \ell_{2k}\delta_x^{2k} =\ell_0+\ell(\delta_x^2)\,,
\end{displaymath}
for some coefficients~$\ell_{2k}$ and corresponding function~$\ell$.
For example, in application to reaction~diffusion \pde{}s, we expand the diffusion operator as\footnote{Such operator expansions and our formal operator manipulations appear to be little known these days, but they are well established~\cite[e.g.]{npl61, Hildebrand1987}.  The manipulations are valid for infinitely differentiable functions as appropriate to the diffusion-like \pde{}s considered here, but not applicable to systems generating shocks or other singularities that are not the subject of this theorem.  Modern analysis typically prefers to invoke Taylor's Remainder Theorem which avoids requiring infinite differentiability, but the aim here is to prove consistency to arbitrarily high order.}
\begin{equation*}
\partial_x^2=\ell(\delta_x^2)=\left[\frac2H\sinh^{-1}(\rat12\delta_x)\right]^2
=\frac1{H^2}\left[\delta_x^2-\rat1{12}\delta_x^4 
+\rat1{90}\delta_x^6 -\rat1{560}\delta_x^8+\cdots\right],
\end{equation*}
which has the inverse
\begin{equation*}
\delta_x^2=\ell^{-1}(\partial_x^2)=4\sinh^2(H\partial_x/2)
=H^2\partial_x^2 +\rat1{12}H^4\partial_x^4
+\rat1{360}H^6\partial_x^6 +\rat1{20160}H^8\partial_x^8 +\cdots\,.
\end{equation*}
In general, provided the leading coefficient $\ell_2\neq0$\,, the \pde\ $\partial_tu_i=[\ell_0+\ell(\delta_x^2)]u_i$ is equivalently written $\ell^{-1}(\partial_t-\ell_0)u_i=\delta_x^2 u_i$ where $\ell^{-1}$~is the inverse of function~$\ell$ (since $\ell_2\neq 0$\,, $\ell$~is invertible for small enough argument).
 
Evaluate this last form of the \pde\ at $x=X_i$ so that the $u_i$ on the left-hand side becomes simply the grid value~$U_i$ and by the coupling conditions~\eqref{eq:1dibc}\footnote{If the leading coefficient in the expansion of~$\cL$ is $\ell_{2n}\neq0$\,, because the lower order coefficients are zero (or asymptotically small as in the \KS\ \pde), then more coupling conditions like~\eqref{eq:1dibc} couple with the next nearer neighbouring elements.} the centred spatial difference on the right-hand side becomes the centred grid difference~$\gamma\delta^2U_i$\,.
This evaluation then gives the evolution to be $\ell^{-1}(\partial_t-\ell_0)U_i=\gamma\delta^2 U_i$ on the macroscale grid.  

Now reverting the inverse function, this grid evolution is equivalent to
\begin{equation}
\partial_t U_i=\big[\ell_0+\ell(\gamma\delta^2)\big]U_i
=\sum_{k=0}^\infty \ell_{2k}\gamma^k\delta^{2k} U_i\,.
\label{eq:dgis}
\end{equation}
For example, for the diffusion equation
\begin{align*}
\partial_tU_i&=\frac1{H^2}\left[2\sinh^{-1}\big(\rat12\sqrt\gamma\delta\big)\right]^2U_i
\\&{}
=\frac1{H^2}\left[\gamma\delta^2-\frac{\gamma^2}{12}\delta^4
+\frac{\gamma^3}{90}\delta^6 -\frac{\gamma^4}{560}\delta^8 +\cdots\right]U_i\,.
\end{align*}
Thus a truncation of~\eqref{eq:dgis} to errors~$\Ord{\gamma^p}$ results in a discrete model with stencil width of $2p-1$\,.
But specifically relevant to the theorem is the equivalent differential equation of this discrete model evaluated at full coupling:  for smooth macroscale~$U_i$, the error in approximating~$\cL$ by the truncated version of~\eqref{eq:dgis} (at full coupling $\gamma=1$) is dominated by the leading neglected term, namely $\ell_{2p}\delta^{2p}$.
As the element size $H\to 0$, this error is~$\Ord{\ell_{2p}H^{2p}}$ for smooth fields.
For example, for the diffusion operator~$\partial_x^2$, the coefficients $\ell_{2k}=\Ord{1/H^2}$ and so the discrete model is consistent with the diffusion \pde\ to error~$\Ord{H^{2p-2}}$ as grid size $H\to 0$\,.
\end{proof}

The above proof is so slick that one is tempted to immediately apply the identical arguments to nonlinear \pde{}s.
However, the stumbling block is that in the reversion of the operator~$\ell$ we need time derivatives and spatial differences to commute with nonlinearity.
Such commutation is not exact, only approximate, and leads to the following modification of the proof of consistency to nonlinear \pde{}s.

\begin{lemma}[nonlinear consistency]\label{lem:1dc}
  Consider the nonlinear, reaction diffusion \pde\ $\partial_tu=\cL u+g(u)$ for some local, isotropic, homogeneous, conservative, second order, linear operator~$\cL$ and some smooth nonlinear reaction~$g(u)$.
Model the dynamics on overlapping elements of an equi-spaced grid $X_i=iH$\,.
Let $u_i(x,t)$~denote the subgrid field in the $i$th~element satisfying the \pde\ $\partial_tu_i=\cL u_i+g(u_i)$ on the interval~$(X_{i-1},X_{i+1})$ with the moderated inter-element coupling~\eqref{eq:1dibc}.
When inter-element interactions are truncated to some nonlinear order in~$\gamma$, the grid values $U_i(t)=u_i(X_i,t)$, at full coupling $\gamma=1$\,, evolve consistently with the \pde\ $\partial_tu=\cL u+g(u)$. 
\end{lemma}

\begin{proof}
Slightly different to before, the operator~$\cL$ may be formally expanded in even centred differences as
\begin{displaymath}
\cL  =\frac1{H^2}\sum_{k=1}^\infty \ell_{2k}\delta_x^{2k} 
=\frac1{H^2}\ell(\delta_x^2)\,,
\end{displaymath}
for some coefficients~$\ell_{2k}$ and corresponding function~$\ell$.
But here: firstly, since $\cL$~is conservative, the coefficient $\ell_0=0$ as here the reaction~$g(u)$ absorbs any non-zero $\ell_0$; and secondly, since the operator~$\cL$ is second order, $\ell_{2k}=\Ord{1}$ as $H\to 0$\,; the main $H$~dependence is explicit in the above expression.
Noting the inverse $\ell^{-1}(X)=\frac1{\ell_2}X+\Ord{X^2}$ and the \pde\ in each element is $H^2\partial_tu_i=\ell(\delta_x^2)u_i+H^2g(u_i)$, we observe
\begin{align*}
\ell^{-1}(H^2\partial_t) u_i
&{}=\frac1{\ell_2}H^2\partial_t u_i+\Ord{H^4}
\\&{}=\frac1{\ell_2}\left[\ell(\delta_x^2) u_i +H^2g(u_i)\right]+\Ord{H^4}
\\&{}=\delta_x^2 u_i +\frac{H^2}{\ell_2}g(u_i)+\Ord{H^4},
\end{align*}
when the fields $u$~and~$u_i$ and reaction~$g$ are smooth enough so that their $(2k)$th~order differences are~$\Ord{H^{2k}}$ and time derivatives are~$\Ord{1}$ as $H\to0$\,.
Evaluate the above equation at the macroscale grid points $x=X_i$, using the coupling condition~\eqref{eq:1dibc}, to determine
\begin{equation}
\ell^{-1}(H^2\partial_t)U_i 
=\gamma\delta^2U_i 
+\frac{H^2}{\ell_2}g(U_i)+\Ord{H^4}.
\label{eq:1dgrid}
\end{equation}
Use the approximation to the operator~$\ell^{-1}$ to obtain
\begin{align*}
H^2\partial_t U_i
&{}=\ell_2\ell^{-1}(H^2\partial_t)U_i +\Ord{H^4}
\\&{}=\ell_2\left[\gamma\delta^2U_i 
+\frac{H^2}{\ell_2}g(U_i)\right]+\Ord{H^4} \quad\text{from~\eqref{eq:1dgrid}}
\\&{}=\ell_2\gamma\delta^2U_i +H^2g(U_i)+\Ord{H^4}
\\&{}=\ell(\gamma\delta^2)U_i +H^2g(U_i)+\Ord{H^4}.
\end{align*}
Thus, dividing by~$H^2$,
\begin{equation*}
\partial_tU_i=\frac1{H^2}\ell(\gamma\delta^2)U_i+g(U_i)+\Ord{H^2}
\end{equation*}
which, for any truncation of~$\Ord{\gamma^2}$ or higher and then evaluated at full coupling $\gamma=1$\,, is consistent as $H\to 0$ with the nonlinear reaction-diffusion \pde\ $u_t=\cL u+g(u)$.
\end{proof}

We expect that more careful treatment of the $\Ord{H^4}$~error terms in the above will show the error is of higher order in the grid size~$H$, as observed in all examples, but we do not attempt to do so here. 

In this section the subgrid microscale operator~$\cL$ need not be differential.
For example, $\cL$~could be a microscopic lattice operator as in the numerical construction of the previous Section~\ref{chapnumcm}: for example, with $n$~microscopic grid points for every macroscopic grid point we have the microscale lattice diffusion
\begin{equation*}
\delta^2_{\text{micro}}
=4\sinh^2\left[\rat12h\partial_x\right]
=4\sinh^2\left[\frac H{2n}\partial_x\right]
=4\sinh^2\left[\rat1n\sinh^{-1}\rat12\delta_x\right].
\end{equation*}
The proof still holds.
Further, the proof still holds if the subgrid field is not just a scalar.
This last observation empowers the following corollary on general consistency in two dimensions that Section~\ref{S_2D_low} observed specifically for the Ginzburg--Landau \pde.

\begin{corollary}[2D consistency] \label{thm:2dc}
Consider a reaction-diffusion \pde\ $\partial_tu=\nabla^2 u+g(u)$ (such as the Ginzburg--Landau equation~\eqref{E_gl2d}) modelled on overlapping elements with subgrid fields $u_{i,j}(x,y,t)$ coupled by conditions~\eqref{EbcsdL}.
When the interactions are truncated to some nonlinear order in~$\gamma$ the grid values $U_{i,j}(t)=u_{i,j}(X_i,Y_j,t)$, at full coupling $\gamma=1$\,, evolve consistently with the \pde. 
\end{corollary}

\begin{proof}
Apply the previous Lemma~\ref{lem:1dc} twice.
First, treat coordinate~$y$ as a parameter so that the reaction is $\hat g(u)=g(u)+\partial_y^2u$ and $\ell(\delta_x^2)=\partial_x^2$.
Then by Lemma~\ref{lem:1dc} the semi-discrete `grid' values $u_{i,j}(X_i,y,t)$, for discrete~$i$ and parametrised by continuous~$y$, evolve consistently with the reaction-diffusion \pde.
Second, treat index~$i$ as a parameter and consider the discrete modelling in coordinate~$y$ so that now the reaction~$g$ also involves operators acting on the $x$-grid and $\ell(\delta_y^2)=\partial_y^2$.
Then by Lemma~\ref{lem:1dc} the 2D grid values $U_{i,j}=u_{i,j}(X_i,Y_j,t)$ evolve consistently with the semi-discrete system generated in the first step, which in turn is consistent with the diffusion \pde.  
\end{proof}

The generalisation to higher spatial dimensions appears immediate.  

The following adaptation to the dynamics on patches is weaker, but the theorem is a starting result to support the efficient equation-free modelling on patches~\cite[e.g.]{Kevrekidis03b, Samaey03b}.
As centred differences generally scale like $\delta\propto H$ as macrogrid spacing $H\to0$, the theorem confirms simulation errors will be~\Ord{H^{2p-l}} where in the exponent~$l$ characterises the system being simulated ($l=2$ for a diffusion operator).

\begin{theorem}[isotropic patch consistency] \label{thm:ipc}
Consider the \pde\ $\partial_tu=\cL u$ for some local, isotropic, homogeneous, linear operator~$\cL$ in 2D.
Model the dynamics in patches of size $\pat h=rH$ centred on an equi-spaced grid $X_i=iH$ and $Y_j=jH$\,.
Let $u_{i,j}(x,y,t)$~denote the subgrid field in the $(i,j)$th~patch satisfying the \pde\ $\partial_tu_{i,j}=\cL u_{i,j}$ on the patch with the moderated interpatch coupling conditions~\eqref{eq:epcc}.
When interpatch interactions are truncated to residuals~$\Ord{\gamma^p}$ the grid values $\uij(t)=u_{i,j}(X_i,Y_j,t)$, at full coupling $\gamma=1$\,, evolve consistently with the \pde\ $\partial_tu=\cL u$ to errors~\Ord{\delta_i^{2p}+\delta_j^{2p}}.
\end{theorem}

\begin{proof}
First let's establish how patch sized differences relate to grid value differences.
Second, this relates the \pde\ with the evolution of grid values.
Define the patch sized centred difference $\pat\delta_xu(x)=u(x+\pat h/2)-u(x-\pat h/2)$.
Then evaluated at the grid point~$(X_i,Y_j)$ the second centred difference
\begin{align*}
\left[\pat\delta_x^2u_{i,j}\right]_{(X_i,Y_j)}
&{}=\left[\{\shift _x^r-2+ \shift _x^{-r}\}u_{i,j}\right]_{(X_i,Y_j)}
\\&{}=\{\shift _i^r(\gamma)-2+ \shift _i^{-r}(\gamma)\}\uij
\quad\text{by \icc~\eqref{eq:epcc}, and then by~\eqref{eq:shifty}}
\\&{}=\left\{\left[1+\gamma(\mu_i\delta_i+\rat12\delta_i^2)\right]^r
-2+ \left[1+\gamma(-\mu_i\delta_i+\rat12\delta_i^2)\right]^r\right\}\uij
\\&{}=\left\{\left[1+\gamma(\mu_i\delta_i+\rat12\delta_i^2)\right]^r
-2+ \left[1+\gamma(-\mu_i\delta_i+\rat12\delta_i^2)\right]^r\right\}\uij
\\&\quad{}+\Ord{\gamma^p\delta_i^{2p}}
\quad\text{upon truncating to errors }\Ord{\gamma^p}
\\&{}=\left\{\left[1+\mu_i\delta_i+\rat12\delta_i^2\right]^r
-2+ \left[1-\mu_i\delta_i+\rat12\delta_i^2\right]^r\right\}\uij
\\&\quad{}+\Ord{\delta_i^{2p}}
\quad\text{upon evaluating at $\gamma=1$}
\\&{}=\left\{\shift _i^r-2+ \shift _i^{-r}\right\}\uij +\Ord{\delta_i^{2p}}
\\&{}=\pat\delta_i^2\uij +\Ord{\delta_i^{2p}}
\end{align*}
Similarly, the analogous relation holds in the other spatial direction: 
\begin{equation*}
\left[\pat\delta_y^2u_{i,j}\right]_{(X_i,Y_j)}=\pat\delta_j^2\uij +\Ord{\delta_j^{2p}}\,.
\end{equation*}

Now relate the \pde\ and the patch dynamics.
Because the linear operator~$\cL$ is isotropic we may formally write it in terms of the isotropic operator as $\cL=\ell_0+\ell(\pat \delta_x^2+\pat\delta_y^2)$.
Then inverting~$\ell$ the system of \pde{}s on patches, $\partial_tu_{i,j}=\cL u_{i,j}$\,, becomes $\ell^{-1}(\partial_t-\ell_0)u_{i,j}=(\pat \delta_x^2+\pat\delta_y^2)u_{i,j}$\,.
Evaluating at the grid point~$(X_i,Y_j)$ and using the above results leads to $\ell^{-1}(\partial_t-\ell_0)\uij=(\pat \delta_i^2+\pat\delta_j^2)\uij+\Ord{\delta_i^{2p}+\delta_j^{2p}}$.
Lastly, reverting~$\ell$ gives $\partial_t\uij=\big[\ell_0+\ell(\pat \delta_i^2+\pat\delta_j^2)\big]\uij+\Ord{\delta_i^{2p}+\delta_j^{2p}}
=\cL\uij+\Ord{\delta_i^{2p}+\delta_j^{2p}}$.
That is, the grid values evolve consistently with the system simulated within the patches.
\end{proof}

\section{Conclusion}

We explored novel macroscale discretisation of reaction-diffusion dynamics in two spatial dimensions.
This work generalises considerable earlier work on modelling dynamics in one dimension.
The specific coupling conditions~\eqref{EbcsdL} and~\eqref{eq:epcc} have important theoretical and practical consequences for inter-element and interpatch dynamics

Section~\ref{S_2D_divide} discussed how these coupling conditions ensure that centre manifold theory applies to prove the existence, emergence and approximation of the slow manifold that is the macroscale discretisation of general reaction-diffusion equations in two dimensions.
Further, Section~\ref{sec:nccec} proved that the resulting discrete models will also be consistent, as the macroscale grid size $H\to0$\,, with the continuum or microscale dynamics.
Thus the holistic discretisations generated in this novel approach have the dual justification of both consistency for small~$H$ and the existence, by centre manifold theory, of an exact slow manifold at finite~$H$.

This strong theoretical support appears to be straightforwardly generalisable to reaction-diffusion dynamics in three or more dimensions.
The support also appears to be straightforwardly generalisable to higher order \pde{}s, just as the theory supports the discrete modelling of one dimensional higher order \pde{}s such as the \KS\ \pde~\cite{MacKenzie00a, MacKenzie05a}.

Sections \ref{S_2D_low}~and~\ref{chapnumcm}, using the example of the real \glpde, explored technical issues necessary to apply the approach here and to more general \pde{}s.
In contrast to one spatial dimension, a purely algebraic approach can only carried out to a low order of accuracy.
Consequently we introduced and explored an approach where the microscale subgrid dynamics are described numerically, but with algebraic expressions for coefficients so that we construct an algebraic model for the macroscale discretisation.
This work provides a powerful approach and theory for the sound and accurate closure of macroscale simulations.

Section~\ref{sec:patch} showed that the adaptation of this approach to the gap-tooth method of Kevrekidis et al.~\cite{Kevrekidis03b, Samaey03b} is sound, even for very small patches of simulators.  Thus very efficient simulations are possible.  However, further research needs to extend this argument to systems where the microscale dynamics are not the spatially smooth dynamics of reaction-diffusion equations~\eqref{Erde}.

\paragraph{Acknowledgement}  The Australian Research Council Discovery Project grants DP0774311 and DP0988738 helped support this research.  We thank Yannis Kevrekidis for his encouragement and many stimulating discussions.


\end{document}